\documentclass[a4paper,10pt]{amsart}
\usepackage{amsmath,amsthm,amssymb,amsfonts,enumerate,color}
\usepackage[pdftex]{graphicx}

\newcommand{\abs}[1]{\left\vert#1\right\vert}
\newcommand{\set}[1]{\left\{#1\right\}}
\newcommand{\Real}{\mathbb{R}}

\newcommand{\Toron}{\mathbb{T}^n}
\newcommand{\Z}{\mathbb{Z}}
\newcommand{\Zn}{\mathbb{Z}^n}
\newcommand{\N}{\mathbb{N}}

\newcommand{\PV}{\operatorname{P.V.}}
\newcommand{\Dom}{\operatorname{Dom}}
\newcommand{\Id}{\operatorname{I}}

\newcommand{\dive}{\operatornamewithlimits{div}}
\newcommand{\s}{\sigma}

\newtheorem{thm}{Theorem}[section]
\newtheorem{prop}[thm]{Proposition}

\newtheorem{lem}[thm]{Lemma}

\theoremstyle{definition}
\newtheorem{defn}[thm]{Definition}
\newtheorem{rem}[thm]{Remark}

\numberwithin{equation}{section}

\allowdisplaybreaks

\author[L. Roncal]{Luz Roncal}
\address{Departamento de Matem\'aticas y Computaci\'on\\
Universidad de La Rioja\\
26004 Logro\~no, Spain}
\email{luz.roncal@unirioja.es}

\author[P. R. Stinga]{Pablo Ra\'ul Stinga}
\address{Department of Mathematics\\
The University of Texas at Austin\\
1 University Station C1200\\
78712-1202 Austin, TX\\
United States of America}
\email{stinga@math.utexas.edu}

\thanks{This research was partially supported by grants MTM2012-36732-C03-02 
and MTM2011-28149-C02-01 from Spanish Government}



\begin{document}

\title[Fractional Laplacian on the torus]{Fractional Laplacian on the torus}

\begin{abstract}
We study the fractional Laplacian $(-\Delta)^{\sigma/2}$ on the $n$-dimensional torus $\Toron$, $n\geq1$. 
First, we present a general extension problem that describes \textit{any} fractional power $L^\gamma$, $\gamma>0$, where $L$ is a general nonnegative selfadjoint operator defined in an $L^2$-space. This generalizes to all $\gamma>0$ and to a large class of operators the previous known results by Caffarelli and Silvestre. In particular it applies to the fractional Laplacian on the torus. The extension problem is used to prove interior and boundary Harnack's inequalities for $(-\Delta)^{\sigma/2}$, when $0<\sigma<2$.
We deduce regularity estimates on H\"older, Lipschitz and Zygmund spaces.
Finally, we obtain the pointwise integro-differential formula for the operator.
Our method is based on the semigroup language approach.
\end{abstract}

\maketitle


\section{Introduction and statement of results}

Very recently there has been an increasing interest in the study of nonlinear partial differential equations involving fractional operators. Such problems arise naturally in applications like Fluid Dynamics \cite{Caffarelli-Vasseur, D, Kiselev-Nazarov-Volberg}, Strange Kinetics and Anomalous Transport \cite{Shlesinger-Zaslavsky-Klafter}, Financial Mathematics \cite{Caffarelli-Salsa-Silvestre, Silvestre}, among many others. There are some issues in these nonlinear nonlocal fractional problems, not covered by the general theory, in which tools like pointwise formulas, H\"older estimates and Harnack's inequalities are needed \cite{Caffarelli-Salsa-Silvestre, Caffarelli-Vasseur, D, Silvestre, Stinga, Stinga-Zhang}.

We develop a systematic study of the fractional powers of the Laplacian
on the $n$-dimensional torus $\Toron$, $n\geq1$:
$$(-\Delta)^{\sigma/2}\equiv(-\Delta_{\Toron})^{\sigma/2}.$$
 This operator arises in models with periodic boundary conditions,
see for example \cite{D, Kiselev-Nazarov-Volberg}. Our aim is to prove Harnack's inequalities,
regularity estimates and pointwise formulas by using the semigroup language approach.

For an integrable function $f$ on $\Toron$ we write its multiple Fourier series expansion as
$$f(x)=\sum_{\nu\in\Zn}c_\nu(f)e^{i\nu\cdot x},$$
where the Fourier coefficient is defined as
$$c_\nu(f)=\frac{1}{(2\pi)^n}\int_{\Toron}f(x)e^{-i\nu\cdot x}\,dx,$$
and $x\cdot\nu=x_1\nu_1+\cdots+x_n\nu_n$, $x\in\Toron$.
The fractional Laplacian on the torus is given by
\begin{equation}\label{series toro}
(-\Delta)^{\sigma/2}f(x)=\sum_{\nu\in\Zn}|\nu|^{\sigma}c_\nu(f)e^{i\nu\cdot x},\quad x\in\Toron,~
\sigma>0.
\end{equation}
This is a nonlocal operator when $\sigma/2$ is not an integer.
Recall that in \cite{Caffarelli-Silvestre} L. Caffarelli and L. Silvestre showed that the fractional Laplacian on $\Real^n$ can be determined as an operator that maps a Dirichlet boundary condition to a Neumann-type condition via an extension problem. Namely, let $u$ be the solution to the boundary value problem
$$
\begin{cases}
\Delta_{\Real^n}u+\frac{1-\sigma}{\tau}u_\tau+u_{\tau\tau}=0,&\hbox{in}~\Real^n\times(0,\infty),\\
u(x,0)=f(x),&\hbox{on}~\Real^n.
\end{cases}$$
Then there exists a constant $c_\sigma>0$ such that
$$
-\lim_{\tau\to0^+}\tau^{1-\sigma}
u_\tau(x,\tau)=c_\sigma(-\Delta_{\Real^n})^{\sigma/2}f(x),\quad x\in\Real^n.
$$
In this paper
we present an extension result for fractional operators $L^\gamma$, where $\gamma$ is \textit{any} noninteger positive number.\footnote{In this way we answer a question raised by Ricardo G. Dur\'an about
the description of higher-order fractional Laplacians via an extension problem.} In particular, it applies to
any higher power $(-\Delta)^{\sigma/2}$, $\sigma>0$.

\begin{thm}[Higher order extension problem]
\label{thm:Extension problem}
Let $L$ be a nonnegative self-adjoint linear operator on a Hilbert space $L^2(\Omega)$.
Let $\gamma\in(0,\infty)\setminus\N$, and let $f$ be in the domain of $L^{\gamma}$. A solution $u\in C^\infty((0,\infty);\Dom(L))\cap C([0,\infty);L^2(\Omega))$ of the extension problem
\begin{equation}\label{extension}
\begin{cases}
      -L_xu+\frac{1-2\gamma}{\tau}\,u_\tau+u_{\tau\tau}=0, & \hbox{in}~\Omega\times(0,\infty), \\
      u(x,0)=f(x), & \hbox{on}~\Omega,
\end{cases}
\end{equation}
is given by
\begin{equation}
\label{eq:Solution Ext Problem}
u(x,\tau)=\frac{\tau^{2\gamma}}{4^\gamma\Gamma(\gamma)}\int_0^\infty e^{-tL}f(x)\,e^{-\frac{\tau^2}{4t}}\,\frac{dt}{t^{1+\gamma}},
\end{equation}
and
\begin{equation}
\label{eq:limitRecovering}
\lim_{\tau\to0^+}\tau^{1-2(\gamma-[\gamma])}\partial_\tau\left((\tau^{-1}\partial_\tau)^{[\gamma]}u(x,\tau)\right)=\mu_\gamma L^\gamma f(x),
\end{equation}
where $[\gamma]$ is the integer part of $\gamma$ and
$$\mu_\gamma=\frac{4^{\gamma-[\gamma]}\Gamma(\gamma-[\gamma])}{2(\gamma-[\gamma])\Gamma(-(\gamma-[\gamma]))}\cdot \frac{1}{2^{[\gamma]}(\gamma-[\gamma])(\gamma-[\gamma]+1)\cdots(\gamma-1)}.$$
\end{thm}

Notice that in \cite{ST} the Caffarelli--Silvestre extension problem was generalized to apply to fractional operators $L^\gamma$, but only for $0<\gamma<1$.

By using the extension problem, we can prove
interior and boundary Harnack's inequalities for the fractional Laplacian on the torus
when $0<\sigma<2$.

\begin{thm}[Interior Harnack inequality]\label{Thm:interior Harnack}
Let $0<\sigma<2$ and let $\mathcal{O}\subseteq\Toron$ be an open set. Fix a compact subset $K\subset \mathcal{O}$. There exists a positive constant $C$ depending only on $n, \sigma$ and $K$ such that
$$\sup_Kf\leq C\inf_Kf,$$
for all solutions  to
$$\begin{cases}
(-\Delta)^{\s/2}f=0,&\hbox{in}~L^2(\mathcal{O}),\\
f\geq0,&\hbox{on}~\Toron\\
f\in\Dom(-\Delta).
\end{cases}$$
As a consequence, any solution $f$ of the problem above is a H\"older continuous function in $K$.
\end{thm}

\begin{thm}[Boundary Harnack's inequality]\label{Thm:boundary Harnack}
Let $0<\sigma<2$ and $f_1,f_2\in\Dom(-\Delta)$ be two nonnegative functions on $\Toron$. Suppose that $(-\Delta)^{\sigma/2}f_j=0$ in $L^2(\mathcal{O})$, for some open set $\mathcal{O}\subseteq\Toron$. Let $x_0\in \partial \mathcal{O}$ and assume that $f_j=0$  for all $x\in B_r(x_0)\cap \mathcal{O}^c$. Assume also that $\partial \mathcal{O}\cap B_r(x_0)$ is a Lipschitz graph in the direction of $x_1$. Then, there is a constant $C$ depending on $\mathcal{O}$, $x_0$, $r$, $n$ and $\sigma$, but not on $f_1$ or $f_2$, such that
$$\sup_{x\in \mathcal{O}\cap B_{r/2}(x_0)}\frac{f_1(x)}{f_2(x)}\leq C\inf_{x\in \mathcal{O}\cap B_{r/2}(x_0)}\frac{f_1(x)}{f_2(x)}.$$
Moreover, $\frac{f_1}{f_2}$ is $\alpha$-H\"older continuous in $\overline{\mathcal{O}\cap B_{r/2}(x_0)}$, for some universal $0<\alpha<1$.
\end{thm}

We also analyze regularity properties of the fractional Laplacian
on the torus
 on H\"older, Lipschitz and Zygmund spaces. Our idea is to characterize all these spaces of smooth functions with the heat semigroup $e^{t\Delta}$, see Proposition \ref{Prop:lambdas-n}. Then we can take advantage of the semigroup formula for the fractional Laplacian on $\Toron$:
\begin{equation}\label{formula semigrupo calor}
(-\Delta)^{\sigma/2}f(x)=\frac{1}{\Gamma(-\s/2)}
\int_0^{\infty}\big(e^{t\Delta}f(x)-f(x)\big)\,\frac{dt}{t^{1+\sigma/2}}, \quad 0<\sigma<2.
\end{equation}

\begin{thm}[Interaction with H\"older spaces]\label{Thm:Holder-n}
Let $\alpha\in(0,1]$ and $0<\sigma<2$.
\begin{itemize}
    \item[(1)] Let $f\in C^{0,\alpha}(\Toron)$ and $\sigma<\alpha$. Then $(-\Delta)^{\sigma/2}f\in C^{0,\alpha-\sigma}(\Toron)$ and
     $$\|(-\Delta)^{\sigma/2}f\|_{C^{0,\alpha-\sigma}(\Toron)}\leq C\|f\|_{C^{0,\alpha}(\Toron)}.$$
    \item[(2)] Let $f\in C^{1,\alpha}(\Toron)$ and $\sigma<\alpha$. Then $(-\Delta)^{\sigma/2}f\in C^{1,\alpha-\sigma}(\Toron)$ and
     $$\|(-\Delta)^{\sigma/2}f\|_{C^{1,\alpha-\sigma}(\Toron)}\leq C\|f\|_{C^{1,\alpha}(\Toron)}.$$
    \item[(3)] Let $f\in C^{1,\alpha}(\Toron)$ and $\sigma\geq\alpha$, with $\alpha-\sigma+1\neq0$. Then $(-\Delta)^{\sigma/2}f\in C^{0,\alpha-\sigma+1}(\Toron)$ and
     $$\|(-\Delta)^{\sigma/2}f\|_{C^{0,\alpha-\sigma+1}(\Toron)}\leq C\|f\|_{C^{1,\alpha}(\Toron)}.$$
    \item[(4)] Let $f\in C^{k,\alpha}(\Toron)$ and assume that $k+\alpha-\sigma$ is not an integer. Then $(-\Delta)^{\sigma/2}f\in C^{l,\beta}(\Toron)$, where $l$ is the integer part of $k+\alpha-\sigma$ and $\beta=k+\alpha-\sigma-l$.
\end{itemize}
\end{thm}

From \eqref{formula semigrupo calor} we can obtain pointwise formulas.

\begin{thm}[Pointwise formula]\label{thm:puntual}
For $0<\sigma<2$ we define the following positive kernel on $\Toron$:
\begin{equation}\label{kernel}
\begin{aligned}
K^{\sigma/2}(x) &:=\frac{1}{|\Gamma(-\sigma/2)|}\int_0^\infty W_t(x)\,\frac{dt}{t^{1+\sigma/2}} \\
&=\frac{2^\sigma\Gamma(\frac{n+\sigma}{2})}{|\Gamma(-\sigma/2)|\pi^{n/2}}\sum_{\nu\in\Zn}\frac{1}{|x-2\pi\nu|^{n+\sigma}},
\quad x\in\Toron,~x\neq0.
\end{aligned}
\end{equation}
\begin{enumerate}[$(1)$]
\item Suppose that $0<\sigma<1$. If $f\in C^{0,\sigma+\varepsilon}(\Toron)$, for some $\varepsilon>0$ such that $0<\sigma+\varepsilon\le1$, then $(-\Delta)^{\sigma/2}f$ is a continuous function and, for all $x\in\Toron$,
\begin{equation}\label{puntual01}
(-\Delta)^{\sigma/2}f(x)=\int_{\Toron}(f(x)-f(y))K^{\sigma/2}(x-y)\,dy.
\end{equation}
The integral above is absolutely convergent.
\item Suppose that $1\leq\sigma<2$. If $f\in C^{1,\sigma+\varepsilon-1}(\Toron)$, for some $\varepsilon>0$ such that $0<\sigma+\varepsilon-1\le1$, then $(-\Delta)^{\sigma/2}f$ is a continuous function and, for all $x\in\Toron$,
\begin{equation}\label{puntual12}
\begin{aligned}
(-\Delta)^{\sigma/2}f(x)&= \int_{\Toron}(f(x)-f(y)-\nabla f(x)\cdot(x-y))K^{\sigma/2}(x-y)\,dy \\
&=\PV\int_{\Toron}(f(x)-f(y))K^{\sigma/2}(x-y)\,dy.
\end{aligned}
\end{equation}
\end{enumerate}
\end{thm}

By using the nonlocal formulas \eqref{puntual01} and \eqref{puntual12}, we show in Propositions \ref{prop:limit zero-n} and \ref{prop:limit 2-n} that
\begin{equation}\label{limit0}
\lim_{\sigma\to0^+}(-\Delta)^{\sigma/2} f(x)=f(x)-\frac{1}{(2\pi)^n}\int_{\Toron} f(y)\,dy,
\end{equation}
and
\begin{equation}\label{limit2}
\lim_{\sigma\to2}(-\Delta)^{\sigma/2} f(x)=-\Delta f(x),
\end{equation}
in the \textit{pointwise sense} for all $x\in\Toron$. Observe the contrast of \eqref{limit0} with the case of the fractional Laplacian on $\Real^n$, where $\lim_{\sigma\to0^+}(-\Delta_{\Real^n})^{\sigma/2}f(x)=f(x)$, see \cite[Proposition~2.5]{Stinga}.
Notice that the identities in \eqref{limit0} and \eqref{limit2} are obvious as limits in $L^2(\Toron)$. Here we prove that the limits actually hold in the \textit{pointwise} sense for a large class of smooth functions. A crucial step is to compute all the constants in the kernel $K^{\sigma/2}(x)$ exactly.

We would like to stress that the semigroup language we adopt here is the most adequate for our purposes. In particular, it allows us to compute all the constants exactly, to study regularity properties in a simple and general way and to have an explicit solution for the extension problem in terms of the underlying semigroup.
See also \cite{Stinga}.

One may think in a fairly naive way that our results could be obtained by a mere
periodization of the results for the fractional Laplacian on
$\Real^n$. In fact there are some initial obstructions with such an idea.
Notice that smooth or $L^2$ functions on the torus cannot be identified with Schwartz class or $L^2$
functions on $\Real^n$. The kernel of the fractional Laplacian on $\Real^n$ is not integrable,
so its ``periodization'' in principle has just a formal meaning. Also, it is not clear that the fractional Laplacian
acting on periodic functions (that should be understood in some suitable sense)
coincides with the fractional Laplacian on the torus as defined with multiple Fourier series.
On the other hand, the results on the torus would
depend on the known results of $\Real^n$ already proved.
In our recent paper \cite{Roncal-Stinga} we addressed all these questions.
The present paper is self-contained and we build up the theory from scratch.
We do not stand on any previous results for the Euclidean case.

The structure of the paper is as follows. In Section \ref{Section:6} we focus on the
generalization of the extension problem to any positive power $L^\gamma$.
The proofs of interior and boundary Harnack's inequalities
for the fractional Laplacian on the torus are contained in 
Section \ref{Section:Harnack}. In Section \ref{sec:regularity} we prove
the regularity estimates of Theorem \ref{Thm:Holder-n}.
Finally, in Section \ref{Section:2} the pointwise formulas of Theorem \ref{thm:puntual}
 and the limits \eqref{limit0} and \eqref{limit2} are shown.
Throughout this paper the letters $c$ and $C$ denote positive constants that may change at each occurrence.

\section{The general extension problem}\label{Section:6}

Let $L$ be a nonnegative densely defined self-adjoint operator on some space $L^2(\Omega,d\eta)=L^2(\Omega)$. To fix ideas, we take $\Omega$ to be an open subset of, say, $\Real^n$, $n\geq1$, and $d\eta$ a positive measure on $\Omega$. Nevertheless, we can replace this assumption by a more general one, that is, we can take $L$ to be a normal operator acting on an abstract Hilbert space. Indeed, the main analytic tool
we will use in the proof is the spectral theorem. Hence, important examples like Laplace--Beltrami operators on Riemannian manifolds or Lie groups, divergence form elliptic operators on domains of $\Real^n$, pseudo-differential operators of even order, among others, are covered by Theorem \ref{thm:Extension problem}. See also \cite{Stinga,ST,Stinga-Zhang}.

Under the assumptions above
there is a unique resolution of the identity $E$, supported on the spectrum of $L$, such that
$$\langle Lf,g\rangle=\int_0^\infty\lambda\,dE_{f,g}(\lambda),\quad f\in\Dom(L),~g\in L^2(\Omega).$$
Here $dE_{f,g}(\lambda)$ is a regular Borel complex measure of bounded variation. Throughout this section we use the notation $\langle f,g\rangle=\int_\Omega f(x)g(x)\,d\eta(x)$. The heat-diffusion semigroup generated by $L$ is given by
$$\langle e^{-tL}f,g\rangle=\int_0^\infty e^{-t\lambda}\,dE_{f,g}(\lambda),\quad f,g\in L^2(\Omega),~t\geq0.$$
Fix any $\gamma>0$. The fractional operators $L^\gamma$ are defined by
$$\langle L^\gamma f,g\rangle=\int_0^\infty\lambda^\gamma\,dE_{f,g}(\lambda),\quad f\in\Dom(L^\gamma),~g\in L^2(\Omega),$$
with domain
$$\Dom(L^\gamma)=\Big\{f\in L^2(\Omega):\int_0^\infty\lambda^{2\gamma}\,dE_{f,f}(\lambda)<\infty\Big\}\supset\Dom(L^{[\gamma]}).$$

\begin{proof}[Proof of Theorem \ref{thm:Extension problem}]
As in \cite{ST}, \eqref{eq:Solution Ext Problem} means that $u(\cdot,\tau)\in\Dom(L)$ for any $\tau>0$ and
$$\langle u(\cdot,\tau),g(\cdot)\rangle=\frac{\tau^{2\gamma}}{4^\gamma\Gamma(\gamma)}\int_0^\infty\langle e^{-tL}f,g\rangle e^{-\frac{\tau^2}{4t}}\,\frac{dt}{t^{1+\gamma}},$$
for all $g\in L^2(\Omega)$. It is proved in \cite[Theorem~1.1]{ST} that, when $0<\gamma<1$, the function $u$ given in \eqref{eq:Solution Ext Problem} and interpreted as above, is well defined and satisfies \eqref{extension} and \eqref{eq:limitRecovering}. If we consider $\gamma>1$ then
it is easy to see that $u$ as in \eqref{eq:Solution Ext Problem} is well defined and verifies \eqref{extension}. It remains to prove \eqref{eq:limitRecovering}. We proceed by induction on $[\gamma]$. As we have just said, \eqref{eq:limitRecovering} is valid for $[\gamma]=0$. Assume \eqref{eq:limitRecovering} for $j<\gamma<j+1$, $j\in\N$.
Take $f\in\Dom(L^{\gamma+1})$ and $g\in L^2(\Omega)$. Then
\begin{align*}
    \langle(\tau^{-1}\partial_\tau)u(\cdot,\tau),g(\cdot)\rangle &= \frac{1}{4^{(\gamma+1)}\Gamma(\gamma+1)}\int_0^\infty\langle e^{-tL}f,g\rangle(\tau^{-1}\partial_\tau)\left(\tau^{2(\gamma+1)}e^{-\frac{\tau^2}{4t}}\right)\frac{dt}{t^{1+(\gamma+1)}} \\
     &= \frac{\tau^{2\gamma}}{4^{\gamma+1}\Gamma(\gamma+1)}\int_0^\infty\langle e^{-tL}f,g\rangle e^{-\frac{\tau^2}{4t}}\left(\frac{2(\gamma+1)}{t^{2+\gamma}}-\frac{\tau^2}{2t^{3+\gamma}}\right)dt \\
    &= \frac{-2\tau^{2\gamma}}{4^{\gamma+1}\Gamma(\gamma+1)}\int_0^\infty\langle e^{-tL}f,g\rangle\partial_t\big(e^{-\frac{\tau^2}{4t}}t^{-(1+\gamma)}\big)\,dt \\
    &= \frac{2}{4\gamma}\cdot\frac{\tau^{2\gamma}}{4^\gamma\Gamma(\gamma)}\int_0^\infty\langle e^{-tL}Lf,g\rangle e^{-\frac{\tau^2}{4t}}\frac{dt}{t^{1+\gamma}}=:\frac{1}{2\gamma}\langle v(\cdot,\tau),g(\cdot)\rangle.
\end{align*}
Observe that $v$ is a solution to \eqref{extension} with initial data $v(x,0)=Lf(x)$. By the induction hypothesis,
\begin{align*}
    \lim_{\tau\to0^+}&\langle \tau^{1-2((\gamma+1)-(j+1))}\partial_\tau\left((\tau^{-1}\partial_\tau)^{j+1}u(\cdot,\tau)\right),g(\cdot)\rangle \\
     &= \frac{1}{2\gamma}\lim_{\tau\to0^+}\langle \tau^{1-2(\gamma-j)}\partial_\tau\left((\tau^{-1}\partial_\tau)^jv(\cdot,\tau)\right),g(\cdot)\rangle \\
     &= \frac{1}{2\gamma}\frac{4^{\gamma-j}\Gamma(\gamma-j)}{2(\gamma-j)\Gamma(-(\gamma-j))}\cdot \frac{1}{2^j(\gamma-j)(\gamma-j+1)\cdots(\gamma-1)}\langle L^\gamma(Lf),g\rangle \\
     &= \frac{4^{(\gamma+1)-(j+1)}\Gamma((\gamma+1)-(j+1))}{2((\gamma+1)-(j+1))\Gamma(-((\gamma+1)-(j+1)))} \\
     &\quad \times\frac{1}{2^{j+1}((\gamma+1)-(j+1))((\gamma+1)-(j+1)+1)\cdots((\gamma+1)-1)}\langle L^{\gamma+1}f,g\rangle.
\end{align*}
\end{proof}

\begin{rem}
As in \cite{ST}, more properties of this general extension problem could be established, like Poisson formulas, fundamental solutions, Cauchy--Riemann equations, conjugate Poisson kernels and $L^p$ estimates.
\end{rem}

\begin{rem}
We can push further the class of operators $L$ for which Theorem \ref{thm:Extension problem} is valid. The most general extension result we know holds for generators of \textit{integrated semigroups} and can be found in \cite{Gale}. In particular, Theorem 1.1 of \cite{Gale} applies for generators of semigroups in Banach spaces like $L^p$, or operators with complex spectrum. Examples include fractional powers of $i\Delta$ or $\partial^3_{xxx}$. Moreover, in \cite{Gale}, the case of complex powers $L^\gamma$, $\operatorname{Re}\gamma>0$, is considered.
\end{rem}

\section{Interior and boundary Harnack's inequalities}\label{Section:Harnack}

In this section we apply the extension problem to prove interior and boundary Harnack's inequalities for $(-\Delta)^{\sigma/2}$, $0<\sigma<2$.

Consider the fundamental cube
$Q_n:=(-\pi,\pi]^n$.
We identify the $n$-torus with $Q_n$. Through this identification,
$$\int_{\Toron}f(x)\,dx=\int_{Q_n}f(x)\,dx.$$
In this way we have a natural way to identify $L^p$ spaces on $\Toron$, $1\leq p\leq\infty$.

For $\nu=(\nu_1,\ldots,\nu_n)\in\Zn$, set $|\nu|=(\nu_1^2+\cdots+\nu_n^2)^{1/2}$. The heat semigroup generated by $\Delta$ is defined by
\begin{equation}\label{calor}
e^{t\Delta}f(x)\equiv T_tf(x):=\sum_{\nu\in\Zn}e^{-t|\nu|^2}c_\nu(f)e^{i\nu\cdot x},\quad f\in L^2(\Toron),~t\geq0.
\end{equation}
Then $T_tf(x)$ is the solution of the heat equation $\partial_{t}v=\Delta v$ in $\Toron\times(0,\infty)$, with initial condition $v(x,0)=f(x)$ on $\Toron$. We have the convolution formula
\begin{equation}
\label{eq:heat semigroup}
T_tf(x)=\int_{\Toron}W_t(x-y)f(y)\,dy,\quad x\in\Toron,
\end{equation}
where, for $x\in\Toron$ and $t>0$, the heat kernel on $\Toron$ is given by
\begin{equation}\label{eq:Nucleo calor}
W_t(x)= \frac{1}{(2\pi)^n}\sum_{\nu\in\Zn}e^{-t|\nu|^2}e^{i\nu\cdot x}
=\frac{1}{(4\pi t)^{n/2}}\sum_{\nu\in\Zn}e^{-\frac{|x-2\pi\nu|^2}{4t}}.
\end{equation}

Let us see first that the extension problem for the fractional Laplacian on the torus admits a classical solution. We show this by using the classical Fourier method.

Take $f\in\Dom(-\Delta)$. We first claim that a solution $u:\Toron\times[0,\infty)\to\Real$ to the extension problem \eqref{extension} for $f$ can be written as
\begin{equation}\label{u}
\begin{aligned}
    u(x,\tau) &= \frac{\tau^{\sigma}}{4^{\sigma/2}\Gamma(\sigma/2)}\int_0^\infty e^{t\Delta}f(x)e^{-\frac{\tau^2}{4t}}\,\frac{dt}{t^{1+\sigma/2}} \\
     &= \frac{\tau^{\sigma}}{4^{\sigma/2}\Gamma(\sigma/2)} \sum_{\nu\in\mathbb{Z}^n}c_{\nu}(f)e^{i\nu\cdot x}\int_0^{\infty}e^{-t|\nu|^2}e^{-\frac{\tau^2}{4t}}\,\frac{dt}{t^{1+\sigma/2}}.
\end{aligned}
\end{equation}
Indeed, to see \eqref{u}, observe that the series in \eqref{calor} converges uniformly in $x$, because
\begin{align*}
\sum_{\nu\in\Z^n}e^{-t|\nu|^2}|c_{\nu}(f)|&\le \|f\|_{L^2(\Toron)}\left(\sum_{\nu\in\mathbb{Z}^n}e^{-2t|\nu|^2}\right)^{1/2}
\le C\|f\|_{L^2(\Toron)}\left(\sum_{k\ge0}k^ne^{-2tk^2}\right)^{1/2}\\
&\le C \|f\|_{L^2(\Toron)}t^{-n/2}\left(\sum_{k\ge 0}e^{-ctk^2}\right)^{1/2}\le C \|f\|_{L^2(\Toron)}t^{-n/2-1/4} .
\end{align*}
Since
$$\int_0^\infty\sum_{\nu\in\Z^n}\left|e^{-t|\nu|^2}c_{\nu}(f)e^{-\frac{\tau^2}{4t}}\right|\frac{dt}{t^{1+\sigma/2}}\leq C_f\int_0^\infty e^{-\frac{\tau^2}{4t}}t^{-n/2-1/4}\,\frac{dt}{t^{1+\sigma/2}}<\infty,$$
Fubini's theorem can be applied to obtain the second equality of \eqref{u}.

Secondly, $u(\cdot,\tau)\in C^2(\Toron)$, for every $\tau>0$. Indeed, for $h>0$ and $e_j$ the $j$-th coordinate unit vector in $\Zn$, $j=1,\ldots,n$,
$$\frac{u(x+he_j,\tau)-u(x,\tau)}{h}=\frac{\tau^{\sigma}}{4^{\sigma/2}\Gamma(\sigma/2)} \sum_{\nu\in\mathbb{Z}^n}c_{\nu}(f)\frac{e^{i\nu\cdot(x+he_j)}-e^{i\nu\cdot x}}{h}\int_0^{\infty} e^{-t|\nu|^2}e^{-\frac{\tau^2}{4t}}\,\frac{dt}{t^{1+\sigma/2}}.$$
As
$$\sum_{\nu\in\Z^n}|c_{\nu}(f)|\int_0^\infty |\nu|e^{-t|\nu|^2}e^{-\frac{\tau^2}{4t}}\,\frac{dt}{t^{1+\sigma/2}}\leq \sum_{\nu\in\Z^n}|c_{\nu}(f)|\int_0^\infty e^{-\frac{t|\nu|^2}{2}}e^{-\frac{\tau^2}{4t}}\,\frac{dt}{t^{1+(\sigma+1)/2}}<\infty,$$
by dominated convergence, $u$ is differentiable with respect to $x$ and the derivative can be taken inside the series in \eqref{u}.  A similar argument for $\nabla_xu$ shows that $u(\cdot,\tau)\in C^2(\Toron)$.

Finally, let us see that, for $\mu_\sigma$ the constant in Theorem \ref{thm:Extension problem},
\begin{equation}\label{lim}
\|\tau^{1-\sigma}u_\tau(x,\tau)\|_{L^2(\Toron)}\to \mu_{\sigma/2}\|(-\Delta)^{\sigma/2}f\|_{L^2(\Toron)},\quad\hbox{as}~\tau\to0^+.
\end{equation}
To prove \eqref{lim} we use \eqref{u}, the cancelation
$$\int_0^\infty e^{-\frac{\tau^2}{4t}}\left(\sigma-\frac{\tau^2}{2t}\right)\frac{dt}{t^{1+\sigma/2}}=0,\quad \tau>0,$$
and dominated convergence, as follows:
\begin{align*}
    \|\tau^{1-\sigma}u_\tau(x,\tau)\|_{L^2(\Toron)}^2 &= \sum_{\nu\in\Z^n}|c_{\nu}(f)|^2\left(\frac{1}{4^{\sigma/2}\Gamma(\sigma/2)}\int_0^\infty e^{-t|\nu|^2}e^{-\frac{\tau^2}{4t}}\left[\sigma-\frac{\tau^2}{2t}\right]\frac{dt}{t^{1+\sigma/2}}\right)^2 \\
     &= \sum_{\nu\in\Z^n}|c_{\nu}(f)|^2\left(\frac{1}{4^{\sigma/2}\Gamma(\sigma/2)}\int_0^\infty (e^{-t|\nu|^2}-1)e^{-\frac{\tau^2}{4t}}\left[\sigma-\frac{\tau^2}{2t}\right]\frac{dt}{t^{1+\sigma/2}}\right)^2 \\
    & \underset{\tau\to0}{\longrightarrow} \sum_{\nu\in\Z^n}|c_{\nu}(f)|^2\left(\frac{\sigma}{4^{\sigma/2}\Gamma(\sigma/2)}\int_0^\infty (e^{-t|\nu|^2}-1)\frac{dt}{t^{1+\sigma/2}}\right)^2 \\
    &= C_{\sigma/2}^2\sum_{\nu\in\Z^n}|\nu|^{2\sigma}|c_{\nu}(f)|^2=C_{\sigma/2}^2\|(-\Delta)^{\sigma/2}f\|_{L^2(\Toron)}^2.
\end{align*}
Let us also note that if $f\geq0$ then $u\geq0$.

\begin{proof}[Proof of Theorem \ref{Thm:interior Harnack}]
Set $\tilde{u}(x,\tau)=u(x,|\tau|)$, $x\in\Toron$, $\tau\in\Real$, where $u$ is as in \eqref{u}. Let us verify that $\tilde{u}$ is a nonnegative weak solution of
\begin{equation}\label{degenerada}
\dive(|\tau|^{1-\sigma}\nabla\tilde{u})=0,\quad\hbox{in}~\mathcal{C}:=\mathcal{O}\times(-R,R),~R>0.
\end{equation}
Indeed, for any $\varphi\in C^\infty_c(\mathcal{O}\times(-R,R))$ and $\delta>0$, by applying the divergence theorem,
\begin{align*}
    \int_{\mathcal{C}}|\tau|^{1-\sigma}\nabla\tilde{u}\cdot\nabla\varphi\,dx\,d\tau &= \int_{\mathcal{C}\cap\{|\tau|\geq\delta\}}\dive(|\tau|^{1-\sigma}\varphi\nabla\tilde{u})\,dx\,d\tau+ \int_{\mathcal{C}\cap\{|\tau|<\delta\}}|\tau|^{1-\sigma}\nabla\tilde{u}\cdot\nabla\varphi\,dx\,d\tau \\
     &= \int_{\mathcal{O}}\varphi(x,\delta)\delta^{1-\sigma}\tilde{u}_\tau(x,\delta)\,dx+ \int_{\mathcal{O}\times(-\delta,\delta)}|\tau|^{1-\sigma}\nabla\tilde{u}\cdot\nabla\varphi
     \,dx\,d\tau.
\end{align*}
The first term above is bounded by $\|\varphi\|_{L^\infty((-R,R);L^2(\mathcal{O}))}
\|\delta^{1-\sigma}\tilde{u}_\tau(x,\delta)\|_{L^2(\mathcal{O})}$, which tends to 0 as $\delta\to0^+$ because of \eqref{lim}. As for the second term, we write $\nabla\tilde{u}\cdot\nabla\varphi=\sum_{k=1}^n\partial_{x_k}\tilde{u}
\partial_{x_k}\varphi+\partial_\tau\tilde{u}\partial_\tau\varphi$, so the integral splits into $\sum_{k=1}^nJ_k+J$. To deal with $J_k$, we see that, as $f\in\Dom(-\Delta)$, the derivative $\partial_{x_k}f\in L^2(\Toron)$. Next we check that $\|\partial_{x_k}u(x,\tau)\|_{L^2(\Toron)}\to\|\partial_{x_k}f\|_{L^2(\Toron)}$, as $\tau\to0^+$. This is proved by using \eqref{u}, a change of variables and dominated convergence:
\begin{align*}
    \|\partial_{x_k}u(x,\tau)\|^2_{L^2(\Toron)} &= \sum_{\nu\in\mathbb{Z}^n}\nu_k^2|c_{\nu}(f)|^2\left(\frac{\tau^{\sigma}}{4^{\sigma/2}
    \Gamma(\sigma/2)}\int_0^{\infty} e^{-t|\nu|^2}e^{-\frac{\tau^2}{4t}}\,\frac{dt}{t^{1+\sigma/2}}\right)^2 \\
     &=  \sum_{\nu\in\mathbb{Z}^n}\nu_k^2|c_{\nu}(f)|^2 \left(\frac{1}{\Gamma(\sigma/2)}\int_0^{\infty}e^{-\frac{\tau^2}{4s}|\nu|^2}e^{-s}
     \,\frac{ds}{s^{1-\sigma/2}}\right)^2 \\
     &\longrightarrow \sum_{\nu\in\Z^n}\nu_k^2|c_{\nu}(f)|^2=\|\partial_{x_k}f\|^2_{L^2(\Toron)}
     ,\quad\hbox{as}~\tau\to0^+.
\end{align*}
Thus there exists a constant $C(f)$ such that $\|\partial_{x_k}\tilde{u}(x,\tau)\|_{L^2(\Toron)}<C(f)$ for all sufficiently small $\tau$. Hence, $|J_k|\leq C(f,\varphi)\delta^{2-\sigma}\to0$, as $\delta\to0$. In order to estimate $J$, by using \eqref{lim}, there exists $C$ such that $\||\tau|^{1-\sigma}\tilde{u}_\tau(x,\tau)\|_{L^2(\mathcal{O})}\leq    C$ for all sufficiently small $\tau$. Therefore,
$$|J|\leq \int_{-\delta}^\delta\||\tau|^{1-\sigma}\tilde{u}_\tau(x,\tau)\|_{L^2(\mathcal{O})}\|\partial_\tau\varphi\|_{L^2(\mathcal{O})}\,d\tau\leq C_\varphi\delta\to0,\quad\delta\to0.$$
Hence, $\tilde{u}$ is a nonnegative weak solution to \eqref{degenerada} in $\mathcal{C}=\mathcal{O}\times(-R,R)$. The equation in \eqref{degenerada} is a degenerate elliptic equation with $A_2$ weight $|\tau|^{1-\sigma}$. By applying Harnack's inequality of \cite[Theorems~2.3.8~and~2.3.12]{Fabes-Kenig-Serapioni} to $\tilde{u}$, we get the conclusions for $f$.
\end{proof}

\begin{rem}
In view of Theorems \ref{thm:Extension problem} and \ref{Thm:interior Harnack}, a natural question that arises is how to apply the extension problem to get interior Harnack's inequality for $(-\Delta)^{\sigma/2}$ with $\sigma>2$. First, we must note that some extra hypotheses on $f$ should be added. Indeed, Harnack's inequality for the biharmonic operator $(-\Delta_{\Real^n})^2$ holds if we also know that $(-\Delta_{\Real^n})f\geq0$, the counterexample being $f(x)=x_1^2$ in $B_2(0)$.
Secondly, if $\sigma>2$, the degeneracy weight in the extension equation $\dive(\tau^{1-\sigma}\nabla u)=0$ does not belong to any $A_p$ class and, up to our knowledge, Harnack's inequality in this case is not known.
\end{rem}

\begin{proof}[Proof of Theorem \ref{Thm:boundary Harnack}]
We take $r=1/2$, the proof for a general $r>0$ is the same. Let $\tilde{u}_j(x,\tau)=u_j(x,|\tau|)$, where $u_j$ is the extension of $f_j$ as in Theorem \ref{thm:Extension problem}. As in the proof of Theorem \ref{Thm:interior Harnack}, $\tilde{u}_j$ satisfies the degenerate elliptic equation $\dive(|\tau|^{1-\sigma}\nabla\tilde{u}_j)=0$ in the weak sense in $\mathcal{O}\times\Real$. Moreover, $\tilde{u}_j$ verifies the equation in the weak sense in $(\Toron\times\Real)\setminus\set{(x,0):x\in \mathcal{O}^c}$ and $\tilde{u}_j(x,0)=f_j(x)$ for all $x\in\mathcal{O}^c$. Let us take a bilipschitz map $\Psi:\Toron\rightarrow\Toron$ that flattens $\partial\mathcal{O}\cap B_{1/2}(x_0)$, that is, such that $\Psi(x_0)=0$ and $\Psi(\mathcal{O})\cap B_{1/2}(0)=\set{x_1>0}\cap B_{1/2}(0)$. We can extend this map to $\mathcal{O}\times \Real$ as a constant in the variable $\tau$. Then, the functions $v_j=\tilde{u}_j\circ \Psi^{-1}$ are also solutions of an equation in the same class, namely, $\dive(|\tau|^{1-\sigma}\mathcal{B}\nabla v_j)=0$ in $(\Toron\times \Real)\setminus\{(x,0): x\in\Psi(\mathcal{O})^c\}$. Indeed, for any test function $\varphi\in C_c^{\infty}( \mathcal{O})$,
$$
\int(\nabla_x \varphi)^T\nabla_x \tilde{u}_j\,dx=\int(\nabla\psi)^T(D \Psi)^T(D\Psi)\nabla v_j\,\frac{dz}{\det D\Psi},
$$
where $\psi(z)=\varphi(\Psi^{-1}(z))$ and $D\Psi$ denotes the Jacobian matrix of the transformation. Then matrix $\mathcal{B}$ is $\tfrac{(D\Psi)^T(D\Psi)}{\det D\Psi}$, which is uniformly elliptic because $\Psi$ is a bilipschitz transformation. See Figure \ref{figure1}.

\begin{figure}[here]
\includegraphics[width=0.7\textwidth]{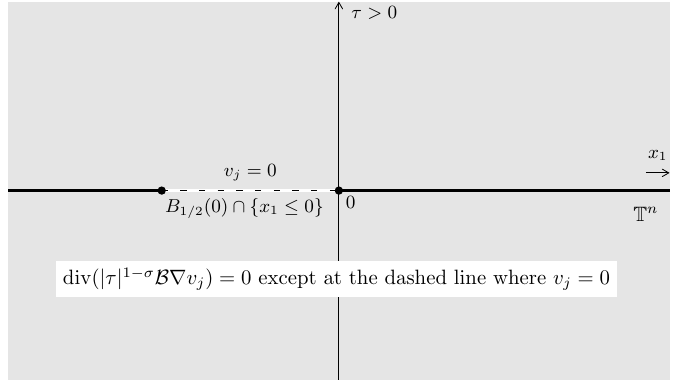}
\caption{The equation for $v_j$}
\label{figure1}
\end{figure}

For $(x,\tau)=(x_1,\ldots,x_n,\tau)\in\Toron\times\Real$ we can write, by using polar coordinates, $(x_1,x_2,\ldots,x_n,\tau)=(\rho\cos\theta,x_2,\ldots,x_n,\rho\sin\theta), \rho>0,~ \theta\in(-\pi,\pi)$.
Consider now the map
$$
\Phi:(\Toron\times\Real)\setminus\set{(x,0):x_1\leq0}\rightarrow (\Toron\times\Real)\cap\set{(x,\tau):x_1>0},
$$
defined to be constant in the variables $x_2,\ldots,x_n$ and such that
$$(\rho\cos\theta,x_2,\ldots,x_n,\rho\sin\theta)\overset{\Phi}{\longmapsto}
(\rho\cos\tfrac{\theta}{2},x_2,\ldots,x_n,
\rho\sin\tfrac{\theta}{2})=:(X_1,x_2,\ldots,x_n,Y)=:(X,Y).
$$
We see that
$$
D\Phi=\left(\begin{array}{ccccc}
\partial_{\rho}X_1&\partial_{x_2}X_1 &\cdots&\partial_{x_n}X_1&\partial_{\theta}X_1\\
\partial_{\rho}x_2& \partial_{x_2}x_2&&&\partial_{\theta}x_2\\
\vdots& &\ddots&&\vdots\\
\partial_{\rho}x_n& &&\partial_{x_n}x_n&\partial_{\theta}x_n\\
\partial_{\rho}Y& \partial_{x_2}Y&\cdots&\partial_{x_n}Y&\partial_{\theta}Y
\end{array}\right)=\left(\begin{array}{ccccc}
\cos\tfrac{\theta}{2}&0 &\cdots&0&-\rho/2\sin\tfrac{\theta}{2}\\
0&1 &&&0\\
\vdots& &\ddots&&\vdots\\
0& &&1&0\\
\sin\tfrac{\theta}{2}&0 &\cdots&0&\rho/2\cos\tfrac{\theta}{2}
\end{array}
\right).
$$
Therefore, if we denote by $I_n$ the identity matrix of size $n\times n$,
$\displaystyle
(D\Phi)^T(D\Phi)=\left(\begin{array}{cc}
I_n& 0 \\
0& \rho^2/4
\end{array}
\right).
$
Then, the singular values of $D\Phi$ are equal to one, except for the one in the direction of $\partial_{\theta}$, that is $\rho/2$. Also, $\det D\Phi=\rho/2=\sqrt{X_1^2+Y^2}/2$. Define $w_j=v_j\circ\Phi^{-1}$, in $\Toron\times \Real\setminus\set{(x,\tau):x_1\leq0}$. Then $w_j$ is a nonnegative weak solution of $\dive(\mathcal{C}\nabla w_j)=0$, in $B_{1/2}(0)\cap\{(x,\tau):x_1>0\}$. Here $\mathcal{C}=\frac{(D\Phi)^T\mathcal{B}(D\Phi)}{\det D\Phi}\,m(X,Y)$, and $m(X,Y)=\big|\frac{2X_1Y}{\sqrt{X_1^2+Y^2}}\big|^{1-\sigma}$. See Figure \ref{figure2}.

\begin{figure}[here]
\includegraphics[width=0.7\textwidth]{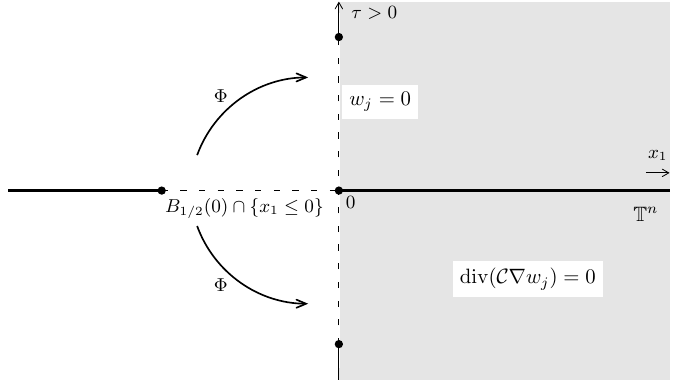}
\caption{The equation for $w_j$}
\label{figure2}
\end{figure}

The equation for $w_j$ above is a degenerate elliptic equation with $A_2$ weight. Therefore, we can apply the theory in \cite[Section 2]{Fabes-Kenig-Jerison} to get boundary Harnack's inequality
$$\sup_{B_{1/4}(0)\cap\{(x,\tau):x_1>0\}}\frac{w_1}{w_2}\leq C\inf_{B_{1/4}(0)\cap\{(x,\tau):x_1>0\}}\frac{w_1}{w_2},$$
and the H\"older continuity of $w_1/w_2$. Go back to $\tilde{u}_1$ and $\tilde{u}_2$ and restrict them to $\tau=0$ for the conclusion.
\end{proof}

\section{Regularity estimates in H\"older spaces}\label{sec:regularity}

\begin{defn}\label{Lambda beta calor}
Let $\beta>0$ and $k=[\beta/2]+1$. We define
$$
\Lambda_\beta(\Toron):=\set{f\in L^\infty(\Toron):\|t^{k}\partial_t^kT_tf(x)\|_{L^\infty(\Toron)}\leq At^{\beta/2},~t>0}.$$
We denote by $[f]_{\Lambda_\beta(\Toron)}$ the least constant $A$ appearing above. The norm in $\Lambda_\beta(\Toron)$ is given by
$$\|f\|_{\Lambda_\beta(\Toron)}=\|f\|_{L^\infty(\Toron)}+[f]_{\Lambda_\beta(\Toron)}.$$
\end{defn}

\begin{defn}[H\"older spaces on $\Toron$]
\label{def:Holder}
Let $0<\alpha\le1$ and $k\in\mathbb{N}_0$. A real function $f$ defined on $\Toron$ belongs to the H\"{o}lder space $C^{k,\alpha}(\Toron)$, if $f\in C^k(\Toron)$ and
$$[D^{\gamma}f]_{C^\alpha(\Toron)}:=\sup_{\begin{subarray}{c}x,y\in\Toron\\x\neq y\end{subarray}}\frac{|D^{\gamma}f(x)-D^{\gamma}f(y)|}{\operatorname{dist}(x,y)^{\alpha}}<\infty,
$$
for each multi-index $\gamma\in \N_0^n$ such that $|\gamma|=k$.
Here $\operatorname{dist}(x,y)$ is the geodesic distance from $x$ to $y$ on $\Toron$.
The norm in $C^{k,\alpha}(\Toron)$ is defined as usual.
\end{defn}

In the following result we relate the spaces $\Lambda_\beta(\Toron)$ with the H\"older spaces $C^{k,\alpha}(\Toron)$ and the Zygmund class $\Lambda_\ast$.

\begin{prop}\label{Prop:lambdas-n}
Let $\beta>0$.
\begin{enumerate}[(i)]
    \item Let $f\in L^\infty(\Toron)$ and $k,\ell>\beta/2$ be two integers. Then the two conditions
        $$\|t^k\partial_t^kT_tf\|_{L^\infty(\Toron)}\leq A_kt^{\beta/2},\quad \|t^\ell\partial_t^\ell T_tf\|_{L^\infty(\Toron)}\leq A_\ell t^{\beta/2},\quad\hbox{for}~t>0,$$
        are equivalent. The least constants $A_k$ and $A_\ell$ that satisfy the inequalities above are comparable.
    \item Let $0<\beta'<\beta$ and $f\in\Lambda_{\beta}(\Toron)$. Then $f\in\Lambda_{\beta'}(\Toron)$
    and
    $$\|f\|_{\Lambda_{\beta'}(\Toron)}\leq \|f\|_{\Lambda_{\beta}(\Toron)}.$$
    \item If $0<\beta<1$ then $\Lambda_\beta(\Toron)=C^{0,\beta}(\Toron)$, with equivalent norms.
    \item We have $\Lambda_1(\Toron)=\Lambda_\ast(\Toron)$, the Zygmund class defined as the set of continuous functions $f$ on $\Toron$ such that $|f(x+h)+f(x-h)-2f(x)|\leq C|h|$, for all $x\in\Toron$ and $h\in\Real^n$. The quantity
        $$\|f\|_{\Lambda_\ast(\Toron)}=\|f\|_{L^\infty(\Toron)}+\sup_{|h|>0}
        \frac{\|f(x+h)+f(x-h)-2f(x)\|_{L^\infty(\Toron)}}{|h|},$$
        is equivalent to $\|f\|_{\Lambda_1(\Toron)}$. Consequently, $C^{0,1}(\Toron)\varsubsetneq\Lambda_1(\Toron)$ and $\|f\|_{\Lambda_1(\Toron)}\leq C\|f\|_{C^{0,1}(\Toron)}$.
    \item If $1<\beta<2$ then $f\in \Lambda_\beta(\Toron)$ if and only if $f$ is differentiable and $\nabla f\in\Lambda_{\beta-1}(\Toron)$. Moreover, $\|f\|_{\Lambda_\beta(\Toron)}$ is equivalent to $\|f\|_{L^\infty(\Toron)}+\|\nabla f\|_{\Lambda_{\beta-1}(\Toron)}$. Similarly, $\Lambda_2(\Toron)=\set{f:\nabla f\in\Lambda_\ast(\Toron)}$.
    \item If $\beta$ is not an integer, then $\Lambda_\beta(\Toron)=C^{[\beta],\beta-[\beta]}(\Toron)$ with equivalent norms. Similarly, for $\beta=j\in\N$, we have
    $$\Lambda_j(\Toron)=\set{f:D^kf\in\Lambda_\ast,\,\hbox{for all}~k=(k_1,\ldots,k_n)\in\mathbb{N}_0\,\hbox{such that}~k_1+\cdots+k_n=j}.$$
\end{enumerate}
\end{prop}

\begin{proof}
Note that \textit{(vi)} follows from \textit{(iii)}--\textit{(v)} by iteration. Item \textit{(iii)} of this proposition in the case when $\Toron$ and $T_t$ are replaced by $\Real^n$ and the heat semigroup on $\Real^n$ is already known. Though we think \textit{(iii)} belongs to the folklore, we provide a proof here for completeness. 

\textit{(i)} This is consequence of the semigroup property of $T_t$ and the following simple estimate
\begin{equation}\label{simple}
|\partial_t^kW_t(x)|\leq C_{n,k}\sum_{\nu\in\Zn}\frac{e^{-c_k|x-2\pi\nu|^2/t}}{t^{n/2+k}}.
\end{equation}
Indeed, assume first that $k>\ell$. Then, with a computation as in \eqref{calor at 1},
\begin{align*}
    |t^k\partial_t^kT_tf(x)| &=
    |t^k\partial_t^{k-\ell}T_{t/2}(\partial_t^\ell T_{t/2}f)(x)|=t^k\abs{\int_{\Toron}\partial_t^{k-\ell}W_{t/2}(x-y)
    \partial_t^\ell T_{t/2}f(y)\,dy} \\
     &\leq Ct^{k+\beta/2-\ell}\int_{Q_n}
     \sum_{\nu\in\Zn}\frac{e^{-c|x-y-2\pi\nu|^2/t}}{t^{n/2+k-\ell}}\,dy= Ct^{\beta/2}.
\end{align*}
Suppose now that $k<\ell$. Let $m$ be the integer for which $k<\ell=k+m$. Then,
\begin{align*}
    |t^k\partial_t^k T_tf(x)| &\leq
    t^k\int_t^\infty\int_{s_1}^\infty\cdots\int_{s_{m-1}}^\infty\abs{\partial_{s_m}^{m+k}T_{s_m}f(x)}\,ds_{m}\,\cdots\,ds_2\,ds_1 \\
    &\leq Ct^k\int_t^\infty\int_{s_1}^\infty\cdots\int_{s_{m-1}}^\infty
    s_m^{\beta/2-(m+k)}ds_{m}\,\cdots\,ds_2\,ds_1= Ct^{\beta/2}.
\end{align*}

The conclusion in \textit{(ii)} follows from \textit{(i)} and the observation that the estimate $\|t^{k}\partial_t^kT_tf(x)\|_{L^\infty(\Toron)}\leq At^{\beta/2}$ is relevant only for $t$ near zero (for $t$ large we have a stronger inequality that follows from \eqref{simple}).

For \textit{(iii)}, suppose that $f\in\Lambda_\beta(\Toron)$. 
It is enough to show that for almost every $x$ we have the estimate $|f(x+h)-f(x)|\leq C|h|^\alpha$,
for all $h$. Indeed, a usual regularization argument (see \cite[p.~141]{Stein-Singular}) gives the continuity of $f$ and thus $f\in C^{0,\alpha}(\Toron)$.
For almost every $x$ we can write
$$|f(x)-f(x+h)|\leq|f(x)-T_{|h|^2}f(x)|+|T_{|h|^2}f(x)-T_{|h|^2}f(x+h)|
+|T_{|h|^2}f(x+h)-f(x+h)|.$$
Then, since $T_0f(x)=f(x)$, the first term above is bounded by
$$\int_0^{|h|^2}|\partial_sT_sf(x)|\,ds\leq \|f\|_{\Lambda_\beta(\Toron)}\int_0^{|h|^2}s^{-1+\beta/2}\,ds=C\|f\|_{\Lambda_\beta(\Toron)}|h|^{\beta}.$$
The third term is estimated analogously. For the second term, we need to show that
\begin{equation}\label{gradiente}
\|\nabla T_tf(\cdot)\|_{L^\infty(\Toron)}\leq C\|f\|_{\Lambda_\beta(\Toron)}t^{\beta/2-1/2}.
\end{equation}
We see that if \eqref{gradiente} is true then the second term is bounded by
$$ C \sup_\xi|\nabla T_{|h|^2}f(\xi)||h|
\leq C\|f\|_{\Lambda_\beta(\Toron)}(|h|^2)^{\beta/2-1/2}|h|=C\|f\|_{\Lambda_\beta(\Toron)}|h|^\beta.$$
In order to prove \eqref{gradiente}, observe first that the simple estimate
$$|\nabla W_t(x)|\leq C\sum_{\nu\in\Zn}\frac{e^{-c|x-2\pi\nu|^2/t}}{t^{n/2+1/2}},$$
implies
\begin{equation}\label{norma1 grad}
\|\nabla W_t\|_{L^1(\Toron)}\leq Ct^{-1/2}.
\end{equation}
Since $W_t=W_{t_1}\ast W_{t_2}$, $t=t_1+t_2$, $t_j>0$, we get $T_tf(x)=W_{t_1}\ast T_{t_2}f(x)$. Taking $t_1=t_2=t/2$, we have $\partial_t \nabla T_tf=\nabla W_{t/2}\ast (\partial_s T_sf)|_{s=t/2}$. In this way, \eqref{norma1 grad} and the assumption $\|\partial_tT_tf\|_{L^\infty(\Toron)}\leq \|f\|_{\Lambda_\beta(\Toron)}t^{\beta/2-1}$ give
\begin{equation}\label{a lo Stein}
\|\partial_t\nabla T_tf\|_{L^\infty(\Toron)}\leq C\|f\|_{\Lambda_\beta(\Toron)}t^{\beta/2-3/2}.
\end{equation}
Nevertheless, $\|\nabla T_tf\|_{L^\infty(\Toron)}=\|\nabla W_t\ast f\|_{L^\infty(\Toron)}\leq \|\nabla W_t\|_{L^1(\Toron)}\|f\|_{L^\infty(\Toron)} \leq Ct^{-1/2}\|f\|_{L^\infty(\Toron)}$.
Therefore $\nabla T_t f\rightarrow 0$ as $t\rightarrow \infty$, thus we can write $\nabla T_tf=-\int_t^{\infty}\partial_s\nabla T_sf\,ds$. From here, and in view of \eqref{a lo Stein}, we obtain \eqref{gradiente}.

Next let us assume that $f\in C^{0,\beta}(\Toron)$. Clearly, from \eqref{calor at 1}, we have $\int_{\Toron}\partial_t W_t(x)\,dx=0$. Thus, using \eqref{simple},
\begin{align*}
\|\partial_tT_tf(x)\|_{L^\infty(\Toron)}&\leq C \int_{\Toron}|\partial_tW_t(h)||f(x+h)-f(x)|\,dh\\
&\le C\int_{Q_n}\sum_{\nu\in\Zn}\frac{e^{-c|h-2\pi\nu|^2/t}}{t^{n/2+1}}|f(x+h-2\pi\nu)-f(x)|\,dh\\
&\le C\|f\|_{C^{0,\beta}(\Toron)}\int_{Q_n}\sum_{\nu\in\Zn}\frac{e^{-c|h-2\pi\nu|^2/t}}{t^{n/2+1}}|h-2\pi\nu|^{\beta}\,dh\\
&\le C \|f\|_{C^{0,\beta}(\Toron)} \int_{Q_n}\sum_{\nu\in\Zn}\frac{e^{-\tilde{c}|h-2\pi\nu|^2/t}}{t^{n/2+1}}t^{\beta/2}\,dh= C\|f\|_{C^{0,\beta}(\Toron)}t^{\beta/2-1}.
\end{align*}

For the proof of (\textit{iv}) we need the trivial facts that $\int_{\Toron}\partial_{tt} W_t(x)\,dx=0$ and $\partial_{tt} W_t(x)=\partial_{tt} W_t(-x)$. With these, if $f\in \Lambda_*(\Toron)$, we see that
$$
\partial_{tt}T_tf(x)=\frac12\int_{\Toron}\partial_{tt} W_t(h)(f(x+h)+f(x-h)-2f(x))\,dh,
$$
and so, by \eqref{simple},
\begin{align*}
\|\partial_{tt}T_tf(x)\|_{L^\infty(\Toron)}
&\le C\int_{Q_n}\sum_{\nu\in\Zn}\frac{e^{-c|h-2\pi\nu|^2/t}}{t^{n/2+2}}|f(x+h-2\pi\nu)+f(x-h+2\pi\nu)-f(x)|\,dh\\
&\le C\|f\|_{\Lambda_*(\Toron)}\int_{Q_n}\sum_{\nu\in\Zn}\frac{e^{-c|h-2\pi\nu|^2/t}}{t^{n/2+2}}|h-2\pi\nu|\,dh\\
&\le C \|f\|_{\Lambda_*(\Toron)} \int_{Q_n}\sum_{\nu\in\Zn}\frac{e^{-\tilde{c}|h-2\pi\nu|^2/t}}{t^{n/2+2}}t^{1/2}\,dh= C\|f\|_{\Lambda_*(\Toron)}t^{1/2-2}.
\end{align*}

In order to prove that $\Lambda_1(\Toron)\subset\Lambda_\ast(\Toron)$  in \textit{(iv)}, one can follow the ideas in \cite[Chapter V, Section 4.3, Proposition 8]{Stein-Singular}, by taking the heat semigroup in $\Toron$ instead of the Poisson in $\Real^n$. Let us sketch here the main steps. First we observe that, for a function $F$ with two continuous derivatives,
\begin{equation}\label{diferencias segundas}
\|F(x+h)+F(x-h)-2F(x)\|_{L^\infty(\Toron)}\le C|h|^2\|D^2F\|_{L^\infty(\Toron)}.
\end{equation}
By the inclusion $\Lambda_1(\Toron)\subset\Lambda_{\alpha}(\Toron)$, for $\alpha<1$, proved
in \textit{(ii)}, we have $\|\partial_tT_tf(x)\|_{L^\infty(\Toron)}\le C t^{\alpha/2-1}$, so, in particular,
$t\|\partial_tT_tf(x)\|_{L^\infty(\Toron)}\rightarrow 0$, as $t\rightarrow 0$. Hence, we can write
\begin{equation}\label{identity}
f(x)=T_0f(x)=\int_0^t s\partial_{ss}T_{s}f(x)\,ds-t\partial_tT_tf(x)+T_tf(x).
\end{equation}
However, by following an argument similar to the one in the proof of (\textit{iii}), we can prove that 
$\|\partial_{tt}T_tf(x)\|_{L^\infty(\Toron)}\le C\|f\|_{\Lambda_1(\Toron)}t^{1/2-2}$
 implies the estimates $\|D^2 T_tf\|_{L^\infty(\Toron)}\leq C\|f\|_{\Lambda_1(\Toron)}t^{-1/2}$, and $\|\partial_tD^2 T_tf\|_{L^\infty(\Toron)}\leq C\|f\|_{\Lambda_1(\Toron)}t^{-3/2}$. Therefore, by plugging \eqref{identity} into \eqref{diferencias segundas},
$$
\|f(x+h)+f(x-h)-2f(x)\|_{L^\infty(\Toron)}\le C\|f\|_{\Lambda_1(\Toron)}\left[\int_0^t ss^{1/2-2}\,ds+\big(t\cdot t^{-3/2}+t^{-1/2}\big)|h|^2\right].
$$
Take $t=|h|^2$ and the result follows.

Finally, \textit{(v)} follows analogous ideas from \cite[Chapter V, Section 4.3, Proposition 9]{Stein-Singular}. Indeed, take $f\in\Lambda_{\beta}(\Toron)$. By using the same technique as in items (\textit{iii}) and (\textit{iv}) we have that $\|\partial_{ttt}\nabla T_tf\|_{L^\infty(\Toron)}\le Ct^{\beta/2-3}$ implies the estimate $\|\partial_{tt}\nabla T_tf\|_{L^\infty(\Toron)}\le Ct^{\beta/2-5/2}$. With this, we can prove that $\nabla f\in L^{\infty}(\Toron)$ and $f\in \Lambda_{\alpha-1}(\Toron)$  with the equivalence of the norms, just following the same steps as in \cite[Page 148]{Stein-Singular}. The proof of the converse implication works in the same way. We omit further details.
\end{proof}

It is easy to check that for any $\lambda>0$ and $0<\sigma<2$ we have the integral identity
$$\lambda^{\sigma/2}=\frac{1}{\Gamma(-\sigma/2)}
\int_0^\infty(e^{-t\lambda}-1)\,\frac{dt}{t^{1+\sigma/2}}.$$
Plugging this into \eqref{series toro} with $\lambda=|\nu|^2$ and interchanging the summation with the integration, we get \eqref{formula semigrupo calor} for $f\in C^\infty(\Toron)$.
We take formula \eqref{formula semigrupo calor} as the definition of $(-\Delta)^{\sigma/2}f$
when $f$ is a function in the class $\Lambda_\beta(\Toron)$. In fact, this is the correct definition,
see Section \ref{Section:2}.
Taking into account Proposition \ref{Prop:lambdas-n}, we readily see that Theorem \ref{Thm:Holder-n} is a direct corollary of
the following result.

\begin{thm}\label{Thm:Lambdas-n}
Let $\beta>0$ and $0<\sigma<2$ with $\sigma<\beta$. If $f\in\Lambda_\beta(\Toron)$ then $(-\Delta)^{\sigma/2}f\in\Lambda_{\beta-\sigma}(\Toron)$, and
$$\|(-\Delta)^{\sigma/2}f\|_{\Lambda_{\beta-\sigma}(\Toron)}\leq C\|f\|_{\Lambda_\beta(\Toron)}.$$
\end{thm}

\begin{proof}
Let $f$ be in $\Lambda_\beta(\Toron)$ and let $0<\sigma<\beta$. We first show that
$(-\Delta)^{\sigma/2}f$ is bounded. Suppose that $\beta<2$. Then, by using \eqref{formula semigrupo calor},
\begin{align*}
	|(-\Delta)^{\sigma/2}f(x)| &\leq c_{\sigma}\left[\int_0^1\int_0^t|\partial_sT_sf(x)|\,ds\,
	\frac{dt}{t^{1+\sigma/2}}+2\|f\|_{L^\infty(\Toron)}\int_1^\infty\,\frac{dt}{t^{1+\sigma/2}}\right] \\
	&\leq C_{\sigma}\left[[f]_{\Lambda_\beta(\Toron)}\int_0^1\int_0^ts^{\beta/2-1}\,ds\,
	\frac{dt}{t^{1+\sigma/2}}+\|f\|_{L^\infty(\Toron)}\right] \leq C_\sigma\|f\|_{\Lambda_\beta(\Toron)}.
\end{align*}
If $\beta\geq2$ then we pick $\sigma<\beta'<2\leq\beta$. By 
Proposition \ref{Prop:lambdas-n}\textit{(ii)}, $f\in \Lambda_{\beta'}(\Toron)$
and, by the computation above,
 $|(-\Delta)^{\sigma/2}f(x)|\leq C_\sigma\|f\|_{\Lambda_{\beta'}(\Toron)}\leq C_\sigma\|f\|_{\Lambda_\beta(\Toron)}$.

Secondly, we have to prove that
$$\|t^k\partial_t^kT_t(-\Delta)^{\sigma/2}f(x)\|_{L^\infty(\Toron)}\leq C\|f\|_{\Lambda_\beta(\Toron)}t^{\frac{\beta-\sigma}{2}},\quad \hbox{for}~k=\big[\tfrac{\beta-\sigma}{2}\big]+1.$$
Suppose that $0<\sigma<1$. By Remark \ref{rem:semigrupo calor},
\begin{align*}
    (-\Delta)^{\sigma/2}f(x) &= \frac{1}{\Gamma(-\sigma/2)}\int_0^\infty\big(T_sf(x)-f(x)\big)
    \,\frac{ds}{s^{1+\sigma/2}} \\
     &=\frac{1}{\Gamma(-\sigma/2)}\left(J_1(x,t)+J_2(x,t)\right),
\end{align*}
where $J_1(x,t)$ denotes the part of the integral running from $0$ to $t$. By using the semigroup property, the hypothesis and Proposition \ref{Prop:lambdas-n}\textit{(i)},
\begin{align*}
    |t^k\partial_t^kT_tJ_1(x,t)| &= \abs{t^k\partial_t^kT_t\int_0^t\int_0^s\partial_rT_rf(x)\,dr\,\frac{ds}{s^{1+\sigma/2}}} \\
     &\leq t^k\int_0^t\int_0^s\abs{\partial_w^{k+1}T_wf(x)\big|_{w=t+r}}\,dr
     \,\frac{ds}{s^{1+\sigma/2}} \\
     &\leq t^k\|f\|_{\Lambda_\beta(\Toron)}\int_0^t\int_0^s(t+r)^{\beta/2-k-1}dr
     \,\frac{ds}{s^{1+\sigma/2}} \\
     &= t^{\beta/2}\|f\|_{\Lambda_\beta(\Toron)}\int_0^t\int_0^{s/t}(1+u)^{\beta/2-k-1}
     \,du\,\frac{ds}{s^{1+\sigma/2}} \\
     &\leq Ct^{\beta/2}\|f\|_{\Lambda_\beta(\Toron)}\int_0^t\frac{s}{t}
     \,\frac{ds}{s^{1+\sigma/2}}= C\|f\|_{\Lambda_\beta(\Toron)}t^{\frac{\beta-\sigma}{2}}.
\end{align*}
On the other hand, by the semigroup property and Proposition \ref{Prop:lambdas-n}\textit{(i)},
\begin{align*}
    |t^k\partial_t^kT_tJ_2(x,t)| &\leq t^k\int_t^\infty\abs{\partial_w^kT_wf(x)\big|_{w=t+s}}\,\frac{ds}{s^{1+\sigma/2}}+ \int_t^\infty\abs{t^k\partial_t^kT_tf(x)}\,\frac{ds}{s^{1+\sigma/2}} \\
     &\leq C\|f\|_{\Lambda_\beta(\Toron)}\left(t^k\int_t^\infty(t+s)^{\beta/2-k}
     \frac{ds}{s^{1+\sigma/2}}+ t^{\frac{\beta-\sigma}{2}}\right) \\
     &= C\|f\|_{\Lambda_\beta(\Toron)}t^{\frac{\beta-\sigma}{2}}
     \left(\int_1^\infty(1+u)^{\beta/2-k}\,\frac{du}{u^{1+\sigma/2}}+1\right)= C\|f\|_{\Lambda_\beta(\Toron)}t^{\frac{\beta-\sigma}{2}}.
\end{align*}

Consider now the situation $1\leq\sigma<2$. We can write $(-\Delta)^{\sigma/2}f=(-\Delta)^{\sigma/2-1/2}(-\Delta)^{1/2}f
=(-\Delta)^{\sigma/2-1/2}R\nabla f$, where $R=\nabla(-\Delta)^{-1/2}$ are the Riesz transforms on $\Toron$. Observe that, by Proposition \ref{Prop:lambdas-n}\textit{(v)}, if $f\in\Lambda_\beta(\Toron)$, $\beta>\sigma\geq1$, then $\nabla f\in\Lambda_{\beta-1}(\Toron)$. Therefore, the result follows from the boundedness of the Riesz transforms on the spaces $\Lambda_{\gamma}(\Toron)$, $\gamma>0$, (see Zygmund \cite[Chapter~III,~(13.29)]{Zygmund} for the one dimensional case, and Calder\'on--Zygmund \cite[Theorem 11]{Calderon-Zygmund} for the multidimensional case), and also from the case just proved above ($0<\sigma/2-1/2<1$).
\end{proof}

\section{Pointwise formula for the fractional Laplacian on $\Toron$}\label{Section:2}

In this section we obtain the pointwise formula for $(-\Delta)^{\sigma/2}f(x)$ when $f$ belongs to the H\"older spaces. We also prove the pointwise limits \eqref{limit0} and \eqref{limit2}.

Let $f\in C^\infty(\Toron)$. For any $N\in\N$ there exists a constant $C_{N,f}$ such that $|c_\nu(f)|\leq C_{N,f}|\nu|^{-N}$, for all $\nu\in\Zn$, $\nu\neq0$. Therefore, the series that defines $(-\Delta)^{\sigma/2}f$ is absolutely convergent and it is a $C^\infty(\Toron)$-function. We also have the symmetry property $\langle(-\Delta)^{\sigma/2}f,g\rangle_{L^2(\Toron)}=\langle f,(-\Delta)^{\sigma/2}g\rangle_{L^2(\Toron)}$, $g\in C^{\infty}(\Toron)$. In fact, the series in \eqref{series toro} converges in $L^2(\Toron)$ whenever $f$ has the property that $\sum_{\nu\in\Zn}\abs{\nu}^{2\sigma}\abs{c_\nu(f)}^2<\infty$, that is, when $f$ is in the Sobolev space $H^\sigma=\Dom((-\Delta)^{\sigma/2})$. This allows us to extend the definition of $(-\Delta)^{\sigma/2}$ to this class.

Consider the test space $C^\infty(\Toron)$ endowed with the family of norms
$$\|\phi\|_k^2:=\|(\Id - \Delta)^k\phi\|_{L^2(\Toron)}^2=\sum_{\nu\in\Zn}
(1+|\nu|^2)^k|c_\nu(\phi)|^2,\quad k\geq1,$$
A real linear functional $S$ on $C^\infty(\Toron)$ is a periodic distribution if it satisfies the following continuity property: if $\phi_j\in C^\infty(\Toron)$, $\|\phi_j\|_k\to0$ as $j\to\infty$ for every $k\in\mathbb{N}$, then $S(\phi_j)\to0$. Note that if $f\in L^1(\Toron)$ then $f$ defines a periodic distribution by $f(\phi)=\int_{\Toron} f\phi$. See Schwartz \cite[Chapter~VII]{Schwartz}. The fractional Laplacian on the torus is a continuous linear operator on $C^\infty(\Toron)$. We remark that this is a difference with respect to the fractional Laplacian on $\Real^n$, which is not continuous on the natural test space for the Fourier transform, namely, the Schwartz class $\mathcal{S}(\Real^n)$, see \cite{Silvestre}.

\begin{lem}\label{lem:periodic distribution}
Suppose that $S$ is a continuous linear operator on $C^\infty(\Toron)$, such that  $\langle S\phi,\psi\rangle_{L^2(\Toron)}=\langle\phi,S\psi\rangle_{L^2(\Toron)}$, for all $\phi,\psi\in C^\infty(\Toron)$, and
$$S\phi(x)=\int_{\Toron}(\phi(x)-\phi(y))K(x-y)\,dy,\quad \phi\in C^\infty(\Toron),~x\in\Toron.$$
Assume that the kernel $K$ above extends to a $2\pi\Zn$-periodic function on $\Real^n$ with
\begin{equation}\label{size}
\abs{K(x)}\leq\frac{C_{n,\gamma}}{\abs{x}^{n+\gamma}},\quad x\in Q_n,
\end{equation}
for some $0\leq\gamma<1$. Let $f\in C^{0,\gamma+\varepsilon}(\Toron)$, with $0<\gamma+\varepsilon\leq1$, $\varepsilon>0$. Then $Sf$ is well defined as a periodic distribution and it coincides with the continuous function
\begin{equation}\label{Sf}
Sf(x)=\int_{\Toron}(f(x)-f(y))K(x-y)\,dy,\quad x\in\Toron.
\end{equation}
\end{lem}

\begin{proof}
By \eqref{size} and the assumption on $f$, the integral in \eqref{Sf} is absolutely convergent. Indeed, for each $x\in\Toron$,
$$\int_{\Toron}|f(x)-f(y)||K(x-y)|\,dy\leq C\int_{Q_n}|x-y|^{\varepsilon-n}dy<\infty.$$
As $f\in L^1(\Toron)$, we can define $Sf$ as a periodic distribution by using the symmetry of $S$, that is, $(Sf)(\phi):=f(S\phi)=\int_{\Toron} fS\phi$, $\phi\in C^\infty(\Toron)$. Let $f_j(x)=T_{1/j}f(x)$, $j\in\N$, $x\in\Toron$, where $T_t$ is the heat semigroup \eqref{eq:heat semigroup}. It is well known that $f_j\in C^\infty(\Toron)$ and that $f_j\to f$, $j\to\infty$, in $L^p(\Toron)$, $1\leq p\leq\infty$ (the latter is a consequence of \cite[Chapter~VII,~Theorem~2.11]{Stein-Weiss}). It is easy to check that $[f_j]_{C^{\gamma+\varepsilon}(\Toron)}\leq[f]_{C^{\gamma+\varepsilon}(\Toron)}$, for all $j$. Now, from the $L^p(\Toron)$-convergence of $f_j$ to $f$, we can see that $Sf_j\to Sf$ as periodic distributions, which is to say $\lim_{j\to\infty}(Sf_j)(\phi)=\lim_{j\to\infty}\int_{\Toron} f_jS\phi=\int_{\Toron} f(S\phi)=(Sf)(\phi)$, for each $\phi\in C^\infty(\Toron)$. Let $\eta>0$ be arbitrary. There exists $\delta>0$ such that
$$C_{n,\gamma}[f]_{C^{\gamma+\varepsilon}(\Toron)}\int_{\abs{x-y}<\delta,\,y\in Q_n}\abs{x-y}^{\varepsilon-n}dy<\frac{\eta}{3}.$$
Then, for all $j$,
$$\abs{\int_{\abs{x-y}<\delta,\,y\in\Toron}(f_j(x)-f_j(y))K(x-y)\,dy}+\abs{\int_{\abs{x-y}<\delta,\,y\in\Toron}(f(x)-f(y))K(x-y)\,dy}<\frac{2}{3}\eta.$$
On the other hand,
\begin{multline*}
\abs{\int_{\abs{x-y}\geq\delta,~y\in\Toron}\Big[\big(f_j(x)-f_j(y)\big)-\big(f(x)-f(y)\big)\Big]K(x-y)\,dy} \\
\leq C\abs{f_j(x)-f(x)}+C\left(\int_{\Toron}\abs{f_j(y)-f(y)}^2dy\right)^{1/2}\leq\frac{\eta}{3},
\end{multline*}
for all sufficiently large $j$, uniformly in $x$ in a compact subset of $\Toron$. Therefore, the right hand side of \eqref{Sf} with $f_j$ converges uniformly on compact subsets of $\Toron$ to the right hand side of \eqref{Sf} with $f$, and the limit is a continuous function. By uniqueness of the limits, \eqref{Sf} holds.
\end{proof}

\begin{proof}[Proof of Theorem \ref{thm:puntual}]
The second identity in \eqref{kernel} follows from \eqref{eq:Nucleo calor}, Tonelli's theorem and the change of variables $|x-2\pi\nu|^2/(4t)=s$. Indeed,
\begin{align*}
\int_0^\infty W_t(x)\,\frac{dt}{t^{1+\sigma/2}} &= \sum_{\nu\in\Zn}\int_0^\infty\frac{e^{-\frac{|x-2\pi\nu|^2}{4t}}}{(4\pi t)^{n/2}}\,\frac{dt}{t^{1+\sigma/2}} \\
&= \frac{2^\sigma}{\pi^{n/2}}\Bigg(\int_0^\infty e^{-s}s^{\frac{n+\sigma}{2}}\frac{ds}{s}\Bigg)
\sum_{\nu\in\Zn}\frac{1}{|x-2\pi\nu|^{n+\sigma}}.
\end{align*}

To prove \eqref{puntual01}, suppose for the moment that $f\in C^\infty(\Toron)$. Recall that
\begin{equation}\label{calor at 1}
T_t1(x) =1,
\end{equation}
for all $x\in\Toron$, $t>0$. Then, by the formula with the heat semigroup in \eqref{formula semigrupo calor},
\begin{equation}\label{prefubini calor}
(-\Delta)^{\sigma/2}f(x)=\frac{1}{|\Gamma(-\sigma/2)|}\int_0^\infty\int_{\Toron}W_t(x-y)(f(x)-f(y))\,dy\,\frac{dt}{t^{1+\sigma/2}}.
\end{equation}
Since $f\in C^\infty(\Toron)$, by Tonelli's theorem and \eqref{kernel},
\begin{equation}\label{cuenta}
\begin{aligned}
\int_0^\infty\int_{\Toron}&|W_t(x-y)(f(x)-f(y))|\,dy\,\frac{dt}{t^{1+\sigma/2}}=
C\sum_{\nu\in\Zn}\int_{Q_n}\frac{|f(x)-f(y+2\pi\nu)|}{|x-y-2\pi\nu|^{n+\sigma}}\,dy\\
&= C\sum_{\nu\in\Zn}\int_{Q_n-2\pi\nu}\frac{|f(x)-f(y)|}{|x-y|^{n+\sigma}}\,dy
= C\int_{\Real^n}\frac{|f(x)-f(y)|}{|x-y|^{n+\sigma}}\,dy.
\end{aligned}
\end{equation}
In the identities above we are identifying $\Toron$ with $Q_n$ and $f$ with its periodic extension.
The last integral in \eqref{cuenta} is absolutely convergent because the periodic extension of $f$ is bounded (which gives integrability at infinity) and H\"older continuous (which gives integrability
when $x$ is close to $y$).
Hence we can apply Fubini's theorem in \eqref{prefubini calor} to obtain \eqref{puntual01} for $f\in C^\infty(\Toron)$.  Observe that for $\nu\neq0$ and $x\in Q_n$, $|x-2\pi\nu|\geq 2\pi|\nu|/2$, so
\begin{equation}\label{key estimate}
\begin{aligned}
0\leq K^{\s/2}(x)&\leq \frac{2^\sigma\Gamma(\frac{n+\sigma}{2})}{|\Gamma(-\sigma/2)|\pi^{n/2}}\Bigg(\frac{1}{|x|^{n+\s}}+\sum_{\nu\neq0}\frac{1}{|\pi\nu|^{n+\s}}\Bigg) \\
&\leq \frac{2^\sigma\Gamma(\frac{n+\sigma}{2})}{\sigma|\Gamma(-\sigma/2)|\pi^{n/2}}\frac{C_n}{|x|^{n+\s}},\quad x\in Q_n.
\end{aligned}
\end{equation}
We have just applied the asymptotic behavior to the Gamma function to see that
\begin{align*}
\sum_{\nu\in\Zn\setminus\{0\}}\frac{1}{|\nu|^{n+\sigma}} &\leq C_n\sum_{k=1}^\infty\frac{1}{k^{1+\sigma}}\frac{\Gamma(k+n)}{\Gamma(k)k^n} \\
&\leq C_n\sum_{k=1}^\infty\frac{1}{k^{1+\sigma}}\leq C_n\sigma^{-1}.
\end{align*}
Therefore we can apply Lemma \ref{lem:periodic distribution} to get \eqref{puntual01}
and the continuity of $(-\Delta)^{\s/2}f$ for $f\in C^{0,\sigma+\varepsilon}(\Toron)$.

Now we establish \eqref{puntual12}. Suppose again that $f\in C^\infty(\Toron)$. Using \eqref{formula semigrupo calor} and \eqref{calor at 1},
\begin{align*}
    (-\Delta)^{\sigma/2}f(x) &= \frac{1}{-\Gamma(-\sigma/2)} \int_0^\infty\int_{\Toron}(f(x)-f(y))W_t(x-y)\,dy\,\frac{dt}{t^{1+\sigma/2}}\\
     &= \frac{1}{-\Gamma(-\sigma/2)}\int_0^\infty\int_{Q_n}(f(x)-f(x-z))W_t(z)\,dz\,\frac{dt}{t^{1+\sigma/2}} \\
     &=  \frac{1}{-\Gamma(-\sigma/2)}\int_0^\infty\int_{Q_n}(f(x)-f(x-z)-\nabla f(x)\cdot z)W_t(z)\,dz\,\frac{dt}{t^{1+\sigma/2}},
\end{align*}
where in the last identity we applied that $\int_{Q_n}z_iW_t(z)\,dz=0$, for all $i=1,\ldots,n$. Since $f\in C^{1,\sigma+\varepsilon-1}(\Toron)$, we have that $\abs{f(x)-f(x-z)-\nabla f(x)\cdot z}\leq C\|f\|_{C^{1,\sigma+\varepsilon-1}(\Toron)}|z|^{\sigma+\varepsilon}$. This and a computation parallel to \eqref{cuenta} allow us to see that the double integral above is absolutely convergent. Therefore, for smooth functions $f$,
$$(-\Delta)^{\sigma/2}f(x)=\int_{\Toron}(f(x)-f(y)-\nabla f(x)\cdot(x-y))K^{\sigma/2}(x-y)\,dy.$$
with $K^{\sigma/2}(x)$ as in \eqref{kernel}. Noticing that the approximation argument in the proof of Lemma \ref{lem:periodic distribution} can be applied also here --one just has to carry on the gradient in the computations--, we get the identity above for any $f\in C^{1,\sigma+\varepsilon-1}(\Toron)$, and the integral is absolutely convergent. For the principal value, note that $\int_{Q_n}z_iK^{\sigma/2}(z)\,dz=0$.
\end{proof}

We now consider the pointwise limits \eqref{limit0} and \eqref{limit2}.
 It can be easily checked that
\begin{equation}\label{asintoticas gamma}
\frac{(\s-2)^{-1}}{\Gamma(-\s/2)}\to\frac{1}{2},~\hbox{as}~\sigma\to2^-,\qquad
\frac{-2/\s}{\Gamma(-\s/2)}\to1,~\hbox{as}~\sigma\to0^+.
\end{equation}

\begin{prop}\label{prop:limit zero-n}
Let $f\in C^{0,\alpha}(\Toron)$, for some $0<\alpha\leq1$. Then, for each $x\in\Toron$,
$$\lim_{\sigma\to0^+}(-\Delta)^{\sigma/2}f(x)=f(x)-\frac{1}{(2\pi)^n}\int_{\Toron} f(y)\,dy.$$
\end{prop}

\begin{proof}
We must check that
$$(-\Delta)^{\sigma/2}f(x)-f(x)+\frac{1}{(2\pi)^n}\int_{\Toron}f(y)\,dy=\int_{\Toron}(f(x)-f(y))\Bigg[K^{\sigma/2}(x-y)-\frac{1}{(2\pi)^n}\Bigg]dy\to0,$$
as $\sigma\to0^+$. Take any $0<\sigma<\alpha$. Let us call $d_\sigma:=-1/\Gamma(-\s/2)>0$. By \eqref{kernel},
\begin{equation}\label{estrella}
\begin{aligned}
    K^{\sigma/2}(x)-&\frac{1}{(2\pi)^n} = d_\s\int_0^1W_t(x)\,\frac{dt}{t^{1+\sigma/2}}+
    d_\s\int_1^\infty W_t(x)\,\frac{dt}{t^{1+\s/2}}-\frac{1}{(2\pi)^n} \\
     &= d_\s\int_0^1W_t(x)\,\frac{dt}{t^{1+\s/2}}+d_\s\int_1^\infty \left(W_t(x)-\frac{1}{(2\pi)^n}\right)\frac{dt}{t^{1+\s/2}}+\frac{1}{(2\pi)^n}\left(\frac{2d_\s}{\s}-1\right) \\
     &=: I_\s+II_\s+III_\s.
\end{aligned}
\end{equation}
As at the beginning of the proof of Theorem \ref{thm:puntual},
\begin{align*}
0 \leq I_\s &= d_\s\frac{2^\sigma}{\pi^{n/2}}\sum_{\nu\in\Zn}\Bigg(\int_{|x-2\pi\nu|^2/4}^\infty e^{-s}s^{\frac{n+\sigma}{2}}
\frac{ds}{s}\Bigg)\frac{1}{|x-2\pi\nu|^{n+\sigma}} \\
&\leq d_\s\frac{2^\sigma}{\pi^{n/2}}\Bigg(\int_0^\infty e^{-s/2}s^{\frac{n+\sigma}{2}}
\frac{ds}{s}\Bigg)\sum_{\nu\in\Zn}\frac{e^{-c|x-2\pi\nu|^2}}{|x-2\pi\nu|^{n+\sigma}}  \\
&= d_\s\frac{2^{\frac{n+3\s}{2}}\Gamma(\frac{n+\s}{2})}{\pi^{n/2}}
\sum_{\nu\in\Zn}\frac{e^{-c|x-2\pi\nu|^2}}{|x-2\pi\nu|^{n+\sigma}}.
\end{align*}
The constant in front of the sum above behaves like $\s/2$ as $\s\to0^+$, see \eqref{asintoticas gamma}.
For $II_\s(x)$ note that, from \eqref{eq:Nucleo calor}, the Fourier coefficient of $W_t(x)$ corresponding to the zero eigenvalue $\nu=(0,\ldots,0)$ is exactly $\frac{1}{(2\pi)^n}$. Then, using that for nonzero $\nu$ we have $|\nu|\geq1$,
$$\abs{W_t(x)-\frac{1}{(2\pi)^n}}\leq C\sum_{\nu\in\Zn\setminus\set{0}}e^{-t|\nu|^2}\leq Ce^{- t/2}
\sum_{\nu\in\Zn\setminus\set{0}}e^{-t|\nu|^2/2}\leq C_ne^{-t/2},\quad t\geq1.$$
Hence, for a constant $C$ independent of $\s$,
$$|II_\s|\leq c_\s C\int_1^\infty e^{-t/2}\,dt=d_\s C.$$
This estimate and \eqref{asintoticas gamma} give that $II_\s\to0$ as $\s\to0^+$.
Also, $III_\s\to0$ as $\s\to0^+$ because of \eqref{asintoticas gamma}.
Collecting terms in \eqref{estrella},
\begin{multline*}
\int_{\Toron}\abs{f(x)-f(y)}\abs{K^{\sigma/2}(x-y)-\frac{1}{(2\pi)^n}}dy \\
\leq d_\s\frac{2^{\frac{n+3\s}{2}}\Gamma(\frac{n+\s}{2})}{\pi^{n/2}}\int_{\Toron}|f(x)-f(y)|
\left[\sum_{\nu\in\Zn}\frac{e^{-c|x-y-2\pi\nu|^2}}{|x-y-2\pi\nu|^{n+\sigma}}\right]dy+
\|f\|_{L^\infty(\Toron)}F(\sigma),
\end{multline*}
where $F(\s)$ is a function of $\s$ (containing the bounds for the $II_\s$ and $III_\s$) that tends to $0$ as $\s\to0^+$. Also, the first term above goes to $0$ as $\s\to0^+$. Indeed, by the smoothness of $f$ and \eqref{asintoticas gamma},
\begin{align*}
    d_\s\int_{\Toron}|f(x)-f(y)|&\left[\sum_{\nu\in\Zn}
    \frac{e^{-c|x-y-2\pi\nu|^2}}{|x-y-2\pi\nu|^{n+\sigma}}\right]dy
    \leq d_\s C_f\sum_{\nu\in\Zn}\int_{Q_n}\frac{|x-y-2\pi\nu|^{\alpha}}{|x-y-2\pi\nu|^{n+\s}}e^{-c|x-y-2\pi\nu|^2}dy \\
    &= d_\s C_f\int_{\Real^n}\frac{e^{-c|x-y|^2}}{|x-y|^{n+\s-\alpha}}\,dy
     = d_\s C_f\tfrac{\Gamma(\tfrac{\alpha-\s}{2})}{2}\to0,\quad\hbox{as}~\sigma\to0^+.
\end{align*}
\end{proof}

\begin{prop}\label{prop:limit 2-n}
Let $f\in C^2(\Toron)$. Then, for each $x\in\Toron$,
$$\lim_{\sigma\rightarrow 2^-}(-\Delta)^{\sigma/2}f(x)=-\Delta f(x).$$
\end{prop}

\begin{proof}
Recall the pointwise formula \eqref{puntual12} in Theorem \ref{thm:puntual}. To shorten the notation, we let $Rf(x,z):=f(x+z)-f(x)-\nabla f(x)\cdot z$, and $\operatorname{\mathbf{sin}}\big(\tfrac{z}{2}\big):=\big(\sin\tfrac{z_1}{2},\ldots,\sin\tfrac{z_n}{2}\big)$. Then, for any $1\leq\sigma<2$,
\begin{align*}
    (-\Delta)^{\sigma/2}f(x) &= -\int_{\Toron}\big(Rf(x,z)-\tfrac12 z^TD^2f(x)z\big)K^{\sigma/2}(z)\,dz \\
    &\quad-\int_{\Toron}\Big(\tfrac12 z^TD^2f(x)z-2\operatorname{\mathbf{sin}}\big(\tfrac{z}{2}\big)^T D^2f(x)\operatorname{\mathbf{sin}}\big(\tfrac{z}{2}\big)\Big)
    K^{\sigma/2}(z)\,dz\\
     &\quad-2\int_{\Toron}\operatorname{\mathbf{sin}}\big(\tfrac{z}{2}\big)^T D^2f(x)\operatorname{\mathbf{sin}}\big(\tfrac{z}{2}\big)K^{\sigma/2}(z)\,dz \\
     &=: J_{1,\sigma}+J_{2,\sigma}+J_{3,\sigma}.
\end{align*}
Let us show first that $J_{1,\sigma}$ and $J_{2,\sigma}$ tend to $0$ as $\sigma\to2^-$. Let $\varepsilon$ be any positive number. Since $f\in C^2(\Toron)$, there exists $\delta=\delta(\varepsilon)>0$ such that $\abs{D^2f(x)-D^2f(y)}<\varepsilon$ for all $y\in\Toron$ with $\abs{x-y}<\delta$. Hence $\abs{Rf(x,z)-\tfrac12  z^TD^2f(x)z}\le C_{n,f}\abs{z}^2\varepsilon$, if $\abs{z}<\delta$. Then, by \eqref{key estimate} and \eqref{asintoticas gamma},
\begin{align*}
|J_{1,\sigma}|&\le\frac{C_{n,f}2^\sigma\Gamma(\frac{n+\sigma}{2})}{|\Gamma(-\sigma/2)|\pi^{n/2}\sigma}\left[\varepsilon \int_{\abs{z}<\delta,\, z\in Q_n}\abs{z}^{2-n-\sigma}\,dz+\int_{\abs{z}>\delta,\, z\in Q_n}|z|^{-n-\sigma}\,dz\right]\\
&\le \frac{C_{n,f}2^{\sigma}\Gamma(\tfrac{n+\sigma}{2})}{|\Gamma(-\sigma/2)|\pi^{n/2}\sigma}\left[\frac{ \delta^{2-\sigma}\varepsilon}{(2-\sigma)}+\frac{1}{\sigma\delta^{\sigma}} \right]\to C\varepsilon, \quad \hbox{as}\quad\sigma \to 2^-.
\end{align*}
Since $\varepsilon$ was arbitrary, $J_{1,\sigma}\to0$ as $\sigma\to2^-$.
Let us continue with $J_{2,\sigma}$. We have
$$J_{2,\sigma}=-\sum_{i=1}^n\frac{\partial^2f}{\partial x_i^2}(x)\int_{\Toron}\left(\frac{1}{2}z_i^2-2\sin^2\big(\tfrac{z_i}{2}\big)\right)K^{\sigma/2}(z)\,dz.$$
By using the Maclaurin series of $\cos z_i$ and \eqref{key estimate},
\begin{align*}
|J_{2,\sigma}|&\leq C_{n,f}\int_{\Toron}\abs{\frac{1}{2}z_i^2-1+\cos z_i}K^{\sigma/2}(z)\,dz\\
&\le \frac{C_{n,f}2^\sigma\Gamma(\frac{n+\sigma}{2})}{|\Gamma(-\sigma/2)|\pi^{n/2}\sigma}\int_{Q_n}|z|^{4-n-\sigma}\,dz\le\frac{C_{n,f}2^\sigma\Gamma(\frac{n+\sigma}{2})}{|\Gamma(-\sigma/2)|\pi^{n/2}\sigma(4-\sigma)}
\end{align*}
and, due to \eqref{asintoticas gamma}, the last expression tends to $0$ as $\sigma \to 2^-$.
We finally prove that $J_{3,\sigma}=-\Delta f(x)$. By taking into account \eqref{kernel}, \eqref{eq:Nucleo calor}, Tonelli's theorem and the orthogonality of the trigonometric system on the torus,
\begin{align*}
J_{3,\sigma}&=-2\sum_{i=1}^n \frac{\partial^2f}{\partial x_i^2}(x)\int_{\Toron}\sin^2\big(\tfrac{z_i}{2}\big)K^{\sigma/2}(z)\,dz\\
&=\frac{2\Delta f(x)}{\Gamma(-\sigma/2)}\int_0^{\infty}\int_{\Toron}\sin^2\big(\tfrac{z_1}{2}\big)
W_t(z)\,dz\,\frac{dt}{t^{1+\sigma/2}} \\
&=\frac{2\Delta f(x)}{\Gamma(-\sigma/2)(2\pi)^n}\int_0^\infty\int_{\Toron}\sin^2\big(\tfrac{z_1}{2}\big)
\left[\sum_{\nu\in\Zn}e^{-t|\nu|^2}e^{i\nu\cdot z}\right]dz\,\frac{dt}{t^{1+\s/2}}\\
&=\frac{2^{n+1}\Delta f(x)}{\Gamma(-\sigma/2)(2\pi)^n}
\int_0^\infty\int_{\Toron}\left(\frac{1-\cos z_1}{2}\right)\prod_{k=1}^n\left(\frac12+\sum_{\nu_k\in\N}e^{-t\nu_k^2}\cos{\nu_kz_k}\right)dz\,\frac{dt}{t^{1+\s/2}}\\
&= \frac{2^{n+1}\Delta f(x)}{\Gamma(-\sigma/2)(2\pi)^n}
\int_0^\infty\int_{\Toron}\frac{1}{2^n}\left(\frac12-e^{-t}\cos^2z_1\right)dz\,\frac{dt}{t^{1+\s/2}}\\
&= \frac{\Delta f(x)}{\Gamma(-\sigma/2)}
\int_0^\infty\big(1-e^{-t}\big)\,\frac{dt}{t^{1+\s/2}}=-\Delta f(x).
\end{align*}
 \end{proof}

\begin{rem}\label{rem:limit 2+}
When $f\in C^{2,\alpha}(\Toron)$, $0<\alpha\leq1$, we also have $\lim_{\sigma \to 2^+} (-\Delta)^{\sigma/2}f(x)=
-\Delta f(x)$. Indeed, for $2<\sigma<3$ we can write $\sigma=2+\epsilon$ for some $\epsilon>0$.  Then, by Proposition \ref{prop:limit zero-n},
$$\lim_{\sigma \to 2^+} (-\Delta)^{\sigma/2}f(x)=\lim_{\epsilon \to 0^+} (-\Delta)^{\epsilon/2}(-\Delta)f(x)=(-\Delta)f(x)-\frac{1}{(2\pi)^n}\int_{\Toron}(-\Delta)f(y)\,dy=-\Delta f(x).$$
This and Proposition \ref{prop:limit 2-n} yield \eqref{limit2}.
\end{rem}

\begin{rem}\label{rem:semigrupo calor}
Under the hypothesis of Theorem \ref{thm:puntual}, formula \eqref{formula semigrupo calor} holds
for functions $f$ in $\Lambda_\beta(\Toron)$.
 Indeed, just write down the kernel in \eqref{puntual01} and \eqref{puntual12} in terms of the heat kernel and apply Fubini's theorem, by taking into account that $\int_{Q_n}(x-y)W_t(x-y)\,dy=0$ in the second case.
\end{rem}

\bigskip

\noindent\textbf{Acknowledgement.} We thank Luis Caffarelli for useful comments that helped us to improve the paper.




\begin{thebibliography}{10}

\bibitem{Caffarelli-Salsa-Silvestre} L. Caffarelli, S. Salsa and L. Silvestre,
{Regularity estimates for the solution and the free boundary of the obstacle problem for the fractional Laplacian},
\textit{Invent. Math.}
\textbf{171} (2008), 425--461.

\bibitem{Caffarelli-Silvestre} L. Caffarelli and L. Silvestre,
{An extension problem related to the fractional Laplacian},
\textit{Comm. Partial Differential Equations}
\textbf{32} (2007), 1245--1260.

\bibitem{Caffarelli-Vasseur} L. Caffarelli and A. Vasseur,
{Drift diffusion equations with fractional diffusion and the quasi-geostrophic equation},
\textit{Ann. of Math. (2)}
\textbf{171} (2010), 1903--1930.

\bibitem{Calderon-Zygmund} A. P. Calder\'on and A. Zygmund,
{Singular integrals and periodic functions},
\textit{Studia Math.}\textbf{14} (1954), 249--271.


\bibitem{D} M. Dabkowski,
{Eventual regularity of the solutions to the supercritical dissipative quasi-geostrophic equation},
\textit{Geom. Funct. Anal.}
\textbf{21} (2011), 1--13.

\bibitem{Fabes-Kenig-Jerison} E. B. Fabes, C. E. Kenig and D. Jerison,
Boundary behavior of solutions to degenerate elliptic equations,
in: Conference on Harmonic Analysis in Honor of Antoni Zygmund, Vol. I, II (Chicago, Ill., 1981), 577--589,
Wadsworth Math. Ser.,
Wadsworth, Belmont, CA, 1983.

\bibitem{Fabes-Kenig-Serapioni} E. B. Fabes, C. E. Kenig and R. Serapioni,
{The local regularity of solutions of degenerate elliptic equations},
\textit{Comm. Partial Differential Equations}
\textbf{7} (1982), 77--116.

\bibitem{Gale} J. E. Gal\'e, P. J. Miana and P. R. Stinga,
{Extension problem and fractional operators: semigroups and wave equations},
\textit{J. Evol. Equ.} 
\textbf{13} (2013), 343--368.

\bibitem{Kiselev-Nazarov-Volberg} A. Kiselev, F. Nazarov and A. Volberg,
{Global well-posedness for the critical 2D dissipative quasi-geostrophic equation},
\textit{Invent. math.}
\textbf{167} (2007), 445--453.

\bibitem{Roncal-Stinga} L. Roncal and P. R. Stinga,
{Transference of fractional Laplacian regularity},
\textit{Special Functions, Partial Differential Equations and Harmonic Analysis. In honor of Calixto P. Calder—n},
in: Springer Proceedings in Mathematics and Statistics \textbf{108},
A. M. Stokolos, C. Georgakis and W. Urbina (Eds.) (2014), 203--212.

\bibitem{Schwartz} L. Schwartz,
\textit{Th\'eorie des Distributions},
(French), Publications de l'Institut de Math\'ematique de l'Universit\'e de Strasbourg, No. IX-X,
Hermann,
Paris, 1966.

\bibitem{Shlesinger-Zaslavsky-Klafter} M. F. Shlesinger, G. M. Zaslavsky and J. Klafter,
{Strange kinetics},
\textit{Nature}
\textbf{363} (1993), 31--37.

\bibitem{Silvestre} L. Silvestre,
Regularity of the obstacle problem for a fractional power of the Laplace operator,
\textit{Comm. Pure Appl. Math.}
\textbf{60} (2007), 67--112.

\bibitem{Stein-Singular} E. M. Stein,
\textit{Singular Integrals and Differentiability Properties of Functions},
Princeton Mathematical Series \textbf{30},
Princeton Univ. Press,
Princeton, New Jersey, 1970.

\bibitem{Stein-Weiss} E. M. Stein and G. Weiss,
\textit{Introduction to Fourier Analysis on Euclidean Spaces},
Princeton Mathematical Series \textbf{32},
Princeton Univ. Press,
Princeton, New Jersey, 1971.

\bibitem{Stinga} P. R. Stinga,
{Fractional powers of second order partial differential operators: extension problem and regularity theory}. PhD thesis, Universidad Aut\'onoma de Madrid, Madrid, 2010.

\bibitem{ST} P. R. Stinga and J. L. Torrea,
Extension problem and Harnack's inequality for some fractional operators,
\textit{Comm. Partial Differential Equations}
\textbf{35} (2010), 2092--2122.

\bibitem{Stinga-Zhang} P. R. Stinga and C. Zhang,
Harnack's inequality for fractional nonlocal equations,
\textit{Discrete Contin. Dyn. Syst.}
\textbf{33} (2013), 3153--3170.

\bibitem{Zygmund} A. Zygmund,
\textit{Trigonometric Series},
Vol. I, II,
Cambridge Univ. Press,
Cambridge-New York-Melbourne, 1977.

\end{thebibliography}
\end{document}